\documentclass[preprint,3p,12pt,authoryear]{elsarticle}

\usepackage{amssymb}
\usepackage{amsthm}

\usepackage{amsmath}
\usepackage{tikz}
\usepackage{cancel}
\usepackage{comment}
\usepackage{framed}
\usepackage{hyperref}

\journal{Annals of Pure and Applied Logic}

\newtheorem{Theorem}{Theorem}[section]

\newtheorem{Lemma}[Theorem]{Lemma}
\newtheorem{Corollary}[Theorem]{Corollary}

\theoremstyle{definition}
\newtheorem{Definition}[Theorem]{Definition}

\newtheorem{Conjecture}[Theorem]{Conjecture}
\newtheorem{Question}[Theorem]{Question}

\newcommand{\renospace}{r.e.}
\newcommand{\re}{r.e.\ }
\newcommand{\turing}{\mathrm{T}}
\newcommand{\Rt}{\mathcal{R}_{\turing}}
\newcommand{\leqt}{\leq_{\turing}}
\newcommand{\length}{\mathrm{length}}
\newcommand{\fin}{\texttt{fin}}
\newcommand{\defeq}{=_{\mathrm{def}}}
\newcommand{\bm}[1]{\mathbf{#1}}

\begin{document}

\begin{frontmatter}



\title{Towards characterizing the $>\omega^2$-fickle\\ recursively enumerable Turing degrees}
\author{Liling Ko}
\ead{ko.390@osu.edu}
\ead[url]{sites.nd.edu/liling-ko/}



\begin{abstract}
    Given a finite lattice $L$ that can be embedded in the recursively enumerable (\renospace) Turing degrees $\langle\Rt,\leqt\rangle$, it is not known how one can characterize the degrees $\mathbf{d}\in\Rt$ below which $L$ can be bounded. Two important characterizations are of the $L_7$ and $M_3$ lattices, where the lattices are bounded below $\bm{d}$ if and only if $\bm{d}$ contains sets of ``\emph{fickleness}'' $>\omega$ and $\geq\omega^\omega$ respectively. We work towards finding a lattice that characterizes the levels above $\omega^2$, the first non-trivial level after $\omega$. We considered lattices that are as ``short'' and ``narrow'' as $L_7$ and $M_3$, but the lattices characterize also the $>\omega$ or $\geq\omega^\omega$ levels, if the lattices are not already embeddable below all non-zero \re degrees. We also considered upper semilattices (USLs) by removing the bottom meet(s) of some previously considered lattices, but the removals did not change the levels characterized. This leads us to conjecture that a USL characterizes the same \re degrees as the lattice it is based on. We discovered three lattices besides $M_3$ that also characterize the $\geq\omega^\omega$-levels. Our search for a $>\omega^2$-candidate therefore involves the lattice-theoretic problem of finding lattices that do not contain any of the four $\geq\omega^\omega$-lattices as sublattices.

\end{abstract}



\begin{keyword}
    recursively enumerable \sep Turing degrees \sep fickleness \sep computable approximation \sep hierarchy \sep lattice \sep embedding


    \MSC 03D25 \sep 03D55
\end{keyword}

\end{frontmatter}


\section*{Acknowledgements}
This work has been partially supported by Peter Cholak's NSF grant DMS-1854136, and also by funding from Rosalie O'Mahony for women in Mathematics in the University of Notre Dame. The author is grateful to Steffen Lempp for valuable discussion, and indebted to Peter Cholak for his support over many years.

\section{Introduction}
A long standing open problem in the study of the recursively enumerable (r.e.) Turing degrees $\langle \Rt,\leqt\rangle$ involves identifying the lattices $\mathcal{L}=\langle L,\leq,\cap,\cup\rangle$ that can be embedded in $\Rt$, and characterizing the degrees below which such $L$ can be bounded. Here, a lattice is a partial order $\leq$ such that every pair of elements $\bm{a},\bm{b}\in L$ has a join $\bm{a}\cup\bm{b}\in L$, which is the lowest upper bound of the pair, and a meet $\bm{a}\cap\bm{b}\in L$, which is the greatest lower bound of the pair. It is known that $\langle \Rt,\leqt\rangle$ forms an upper semilattice (USL) with minimal element $\bm{0}$ and greatest element $\bm{0}'$. As an USL, the join between any pair must exist, but the meet may not~\citep{ambos1984pairs}. By the minimal pair priority argument~\citep{yates1966minimal,lachlan1966lower}, one can construct incomparable pairs whose meet exists, and therefore embed the diamond lattice (Figure~\ref{fig:diamond}) in $\Rt$. One can also generalize the construction to embed below any nonzero \re degree, any lattice that does not contain $N_5$ or $M_3$ (Figure~\ref{fig:N5-131}) as sublattices~\citep{thomason1971sublattices,lachlan1972embedding}. Lattices that avoid $N_5$ and $M_3$ are called \emph{distributive} lattices~\citep{birkhoff1937rings}, because they can be characterized as lattices $L$ whose meets distribute over the joins:
\[(\forall \bm{a},\bm{b},\bm{c} \in L)\;\; [\bm{a}\cap (\bm{b}\cup\bm{c}) =(\bm{a}\cap\bm{b}) \cup(\bm{a}\cap\bm{c})].\]

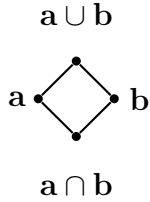
\begin{figure}[tpb]
    \centering
    \begin{tikzpicture}[every node/.style={circle,fill=black,inner sep=1.2pt}]
        \tikzset{diamond/.pic ={
        \node[label=below:{$\bm{a} \cap\bm{b}$}] (d) at (0, -1) {};
        \node[label=left:{$\bm{a}$}] (l) at (-1, 0) {};
        \node[label=above:{$\bm{a} \cup\bm{b}$}] (u) at (0, 1) {};
        \node[label=right:{$\bm{b}$}] (r) at (1, 0) {};
        \draw [-,thick] (u) -- (r) -- (d) -- (l) -- (u);
        }}
        \path (0,0) pic[scale=0.5]{diamond};
    \end{tikzpicture}
    \caption{The diamond lattice can be embedded in $\Rt$~\citep{yates1966minimal,lachlan1966lower}.}
    \label{fig:diamond}
\end{figure}

There is no known procedure for deciding if a given non-distributive lattice can be embedded in $\Rt$. To sketch the challenges faced, consider Birkhoff's representation Theorem~\ref{thm:distributive}:
\begin{Definition} \label{def:join}
    Let $L$ be a finite lattice.
    \begin{enumerate}
        \item An element $\bm{b}\in L$ is \emph{join-irreducible} if $\bm{b}=\bm{a_0}\cup\bm{a_1}$ implies $\bm{b}=\bm{a_0}$ or $\bm{b}=\bm{a_1}$.
        \item An element $\bm{b}\in L$ is \emph{join-prime} if $\bm{b}\leq\bm{a_0}\cup\bm{a_1}$ implies $\bm{b}\leq\bm{a_0}$ or $\bm{b}\leq\bm{a_1}$.
    \end{enumerate}
\end{Definition}

\begin{Theorem}[Birkhoff's Representation \citep{birkhoff1937rings}] \label{thm:distributive}
    Non-distributive lattices must contain a join-irreducible element that is not join-prime.
\end{Theorem}

For $N_5$ and $M_3$, the witnesses to Birkhoff's theorem are shown in Figure~\ref{fig:N5-131} as $\bm{b} \leq\bm{a_0}\cup\bm{a_1}$, which cannot be computed by $\bm{a_0}$ or $\bm{a_1}$. When we construct an \re set $B$ whose degree $\bm{b}$ satisfies these requirements, the diagonalization condition $\bm{b}\nleq\bm{a_0}$ forces us to enumerate elements into $B$. However, the join requirement $\bm{b}\leq\bm{a_0}\cup\bm{a_1}$ says that whenever we enumerate into $B$, we must also enumerate into $A_0$ or $A_1$. Therefore, at least two permissions are requested --- one for $B$ and another for $A_0$ or $A_1$. Distributive lattices work to the effect of involving no joins, so a single permission is enough, allowing any nonzero \re degree to bound the lattice. In non-distributive lattices, joins prevent sets from being constructed independently. These dependencies can become complicated or even contradictory, preventing some lattices, such as $S_8$~\citep{lachlan1980not} and $L_{20}$\citep{lempp1997finite} (Figure~\ref{fig:s8-l20}), from being embedded at all.

The important non-distributive lattices that are known to be embeddable include $N_5$, $M_3$, and $L_7$. The ability for an \re degree $\bm{d}$ to bound these lattices depends only on the ``\emph{fickleness}'' of the sets in $\bm{d}$, with degrees that bound $N_5$ ($L_7$, $M_3$) being exactly those that contain sets of fickleness $>1$ ($>\omega$, $\geq\omega^\omega$)~\citep{downey2007totally,ambos2019universally,downey2020hierarchy}. Here, set fickleness can be thought of as the number of times elements change their minds on their membership in the set. For instance, \re sets have fickleness $\leq1$ since elements are allowed to enter but not exit, and in a similar vein, $n$-\re sets have fickleness $\leq n$. Sets in an \re degree that are not \re may have fickleness $\geq\omega$, and we define a degree's fickleness to be the smallest ordinal $\alpha$ that bounds the fickleness of all sets in the degree. We formalize these definitions in Section~\ref{sec:fickle} \citep{downey2020hierarchy}.

\begin{figure}[tpb]
    \centering
    \tikzset{every picture/.style={thick}}
    \begin{tikzpicture}[every node/.style={circle,fill=black,inner sep=1.2pt}]
        \tikzset{pics/131/.style n args={1}{code ={
            \node (a) at (0, -1) {};
            \node[label=left:{$\bm{a_0}$}] (b) at (-1, 0) {};
            \node (c) at (0, 1) {};
            \node[label=right:{$\bm{a_1}$}] (d) at (1, 0) {};
            \node[label=left:{$\bm{b}$}] (e) at (0, 0) {};
            \draw (a) -- (b) -- (c) -- (d) -- (a) -- (e) -- (c);
            \node[rectangle,align=left,draw=none,fill=none] at (0, -1.5) {$M_3$};
            }}, pics/131/.default={}
        }
        \tikzset{n5/.pic ={
            \node (a) at (0, -1) {};
            \node[label=left:{$\bm{a_0}$}] (b) at (-1, 0) {};
            \node (c) at (0, 1) {};
            \node[label=right:{$\bm{b}$}] (d) at (1, 0.5) {};
            \node[label=right:{$\bm{a_1}$}] (e) at (1, -0.5) {};
            \draw (a) -- (b) -- (c) -- (d) -- (e) -- (a);
            \node[rectangle,align=left,draw=none,fill=none] at (0, -1.5) {#1};
        }}
        \path (-2,0) pic[scale=0.6]{n5=$N_5$};
        \path (2,0) pic[scale=0.6]{131};
    \end{tikzpicture}
    \caption{Non-distributive lattices are those that contain $N_5$ or $M_3$ as sublattices \citep{birkhoff1937rings}.}
    \label{fig:N5-131}
\end{figure}
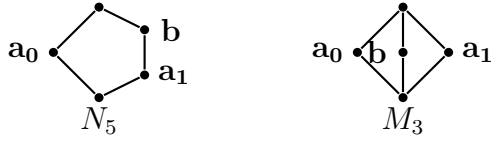

\begin{figure}[tpb]
    \centering
    \tikzset{every picture/.style={thick}}
    \begin{tikzpicture}[every node/.style={circle,fill=black,inner sep=1.2pt}]
        \tikzset{pics/s8/.style n args={1}{code ={
            \node (x) at (-0.5, 1.5) {};
            \node (y) at (0.5, 1.5) {};
            \node (z) at (0, 2) {};
            \node (a) at (0, -1) {};
            \node (b) at (-1, 0) {};
            \node (c) at (0, 1) {};
            \node (d) at (1, 0) {};
            \node (e) at (0, 0) {};
            \draw (a)--(b)--(c)--(d)--(a)--(e)--(c) (c)--(x)--(z)--(y)--(c);
            \node[rectangle,align=left,draw=none,fill=none] at (0, -1.5) {$S_8$};
            }}, pics/s8/.default={}
        }
        \tikzset{pics/lempp/.style n args={1}{code ={
            \node[draw=none,fill=none,align=left] (l) at (0,0.5) {\includegraphics[width=1.5cm]{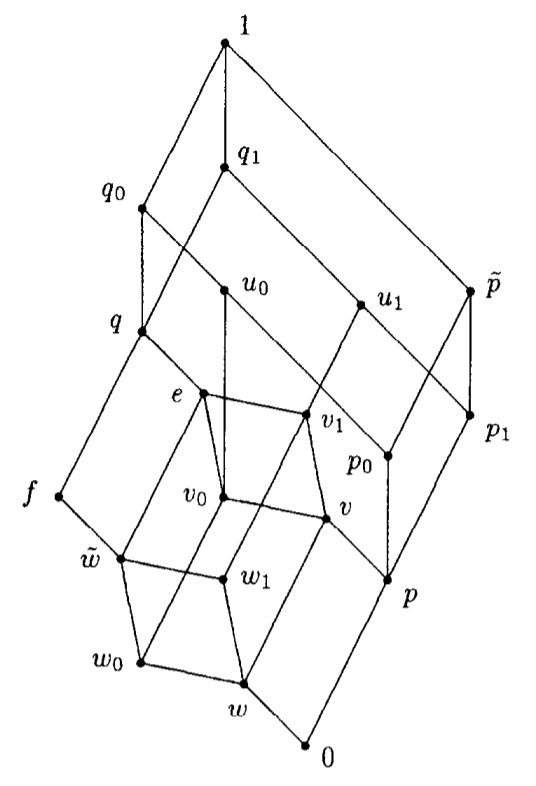}};
            \node[rectangle,align=left,draw=none,fill=none] at (0,-3.0) {$L_{20}$};
            }}, pics/lempp/.default={}
        }
        \path (0,0) pic[scale=0.5]{s8};
        \path (4,0) pic[scale=0.3]{lempp};
    \end{tikzpicture}
    \caption{$S_8$~\citep{lachlan1980not} and $L_{20}$~\citep{lempp1997finite} cannot be embedded in $\Rt$.}
    \label{fig:s8-l20}
\end{figure}

\begin{figure}[tpb]
    \centering
    \tikzset{every picture/.style={thick}}
    \begin{tikzpicture}[every node/.style={circle,fill=black,inner sep=1.2pt}]
        \tikzset{l7/.pic ={
            \node (a) at (0, -1) {};
            \node (b) at (-1, 0) {};
            \node (c) at (0, 1) {};
            \node (d) at (1, 0) {};
            \node (e) at (0, 0) {};
            \draw (a) -- (b) -- (c) -- (d) -- (a);
            \node (f) at (-0.5, -0.5) {};
            \node (g) at (0.5, -0.5) {};
            \draw (c)--(e)--(f) (e)--(g);
            \node[align=left,draw=none,fill=none] at (0,-1.3) {$L_7$};
        }}
        \path (0,0) pic[scale=0.5]{l7};
    \end{tikzpicture}
    \caption{An \re degree bounds $L_7$ if and only if it contains an $>\omega$-fickle set~\citep{downey2007totally,ambos2019universally}.}
    \label{fig:l7}
\end{figure}
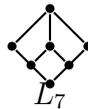

The fickleness hierarchy of the \re degrees was introduced by \citeauthor{downey2020hierarchy} in \citeyear{downey2020hierarchy}, and was based on Ershov's work on $n$-\re sets \citep{ershov1968certain1,ershov1968certain2,ershov1968certain3}. The hierarchy collapses exactly to the powers of $\omega$, making the first few levels at $1$, $\omega$, $\omega^2$, $\ldots$, $\omega^\omega$, $\omega^{\omega+1}$, $\ldots$ \citep{downey2020hierarchy}, where the 1-fickle \re degrees contains just the zero degree. From earlier discussion, all known lattices characterize fickleness at the $>1$, $>\omega$, or $\geq\omega^\omega$-levels, which begs the question of whether there are lattices that characterize the degrees above the second non-trivial level $\omega^2$. This is the motivation of our work and the question remains open.

We search for $>\omega^2$-candidates by considering lattices that are as ``\emph{short}'' and ``\emph{narrow}'' as $L_7$ and $M_3$, like those in Figure~\ref{fig:3direct}. In Section~\ref{sec:sufficient}, we discuss the fickleness that is sufficient for bounding these lattices, and in Section~\ref{sec:necessary} we discuss the fickleness that is necessary. Including $M_3$, we now have four lattices that characterize $\geq\omega^\omega$-fickleness. In Section~\ref{sec:reject}, we use these four to quickly reject ``larger'' candidates, for $>\omega^2$-candidates cannot contain any of those four overly fickle ones as sublattices. In Section~\ref{sec:usl}, we consider some USLs as candidates, obtained by removing the meet(s) of earlier lattices. We show that these USLs characterize the same degrees as their lattices, which leads us to conjecture in \ref{conj:usl-equals-lattice} that all USLs behave this way. We suggest future work in Section~\ref{sec:future}.

\subsection{Defining $\alpha$-fickleness} \label{sec:fickle}
We formalize what it means for an \re set or degree to be $\alpha$-fickle. By Shoenfield's Limit Lemma, every $\Delta^0_2$ set $A$ has a \emph{computable approximation} $a(x,s):\omega^2\to\omega$, such that for all $x,s\in\omega$:
\[A(x) =\lim_s a(x,s).\]

We ``count'' how often $a(x,-)$ changes its mind:
\begin{Definition}[\citep{ershov1968certain1,ershov1968certain2,ershov1968certain3}] ~\label{def:comp-approx}
    Let $\mathcal{R} =(R, <_\mathcal{R})$ be a computable well-ordering (both $R$ and $<_\mathcal{R}$ are computable), and let $A$ be a $\Delta^0_2$ set. An $\mathcal{R}$-computable approximation ($\mathcal{R}$-c.a.) of $A$ is a computable approximation $a(x,s):\omega^2\to\omega$ of $A$ together with a computable \emph{mind-change function} $m(x,s):\omega^2\to R$, such that for all $x$ and $s$:
    \begin{itemize}
        \item $a(x,0)=0$,
        \item $m(x,s+1) \leq_\mathcal{R} m(x,s)$,
        \item if $a(x,s+1) \neq a(x,s)$, then $m(x,s+1) <_\mathcal{R} m(x,s)$.
    \end{itemize}
\end{Definition}

We think of $\alpha$ as the order type of $R$. $a$ is a computable guessing function for $A$, and $m$ associates an ordinal with each guess, so the initial ordinal $m(x,0)$ should represent $\alpha$. Like in the usual approximation for \re sets, the set is empty at first, so our first guess is $a(x,0)=0$ to represent that $x\notin A$. If $a(x,s+1)\neq a(x,s)$, then the ordinal must decrease. After the first mind-change, the ordinal is represented by an element of $R$.

The ideas behind this definition were first formalized by Ershov~\citep{ershov1968certain1,ershov1968certain2,ershov1968certain3}, who referred to a set as $\alpha$-\re if $\alpha$ is the order type of $\mathcal{R}$, and if $\mathcal{R}$ satisfied the additional computable structure of being represented by a notation. Ershov's notion of being $\alpha$-\re was defined for $\alpha<\omega_1^{\mathrm{CK}}$. 
There is an equivalent version given by \citeauthor{epstein1981hierarchies}, and there is another version given by \citeauthor{ash2000computable}.

For the notion of fickleness to be independent from the $\mathcal{R}$ used, even stronger computable structure needs to be imposed on $\mathcal{R}$. To these ends, Downey and Greenberg required $\mathcal{R}$ to be \emph{canonical}, which is stronger than having a notation:
\begin{Definition}[\cite{downey2020hierarchy}] \label{def:canonical}
    Let $\mathcal{R}$ be a computable well-ordering with \emph{order type} $\alpha$. Given $z\in R$, let $|z|$ denote the order type of $z$. Let $\mathrm{nf}_\mathcal{R}: \omega\rightarrow (\omega^2)^{<\omega}$ denote the function that takes each ordinal below $\alpha$ to its Cantor-normal form; i.e. $\forall z\in R$
    \[\mathrm{nf}_\mathcal{R}(z) = \langle \langle z_0, n_0 \rangle, \ldots, \langle z_i, n_i\rangle \rangle,\]
    where $z_j\in R$, $|z_0| > \ldots > |z_i|$, $n_j \in \omega-\{0\}$, and
    \[|z| = \omega^{|z_0|}\cdot n_0 + \ldots + \omega^{|z_i|} \cdot n_i\]
    is the unique Cantor-normal form of $|z|$. If $\mathrm{nf}_\mathcal{R}$ is computable, we say that $\mathcal{R}$ is \emph{canonical}.
\end{Definition}

\begin{Definition}[\cite{downey2020hierarchy}] \label{def:fickleness}
    Let $A,\mathbf{d} \in\Delta^0_2$.
    \begin{enumerate}
        \item $A$ is \emph{$\leq\alpha$-fickle} if $A$ has a canonical $\mathcal{R}$-c.a.\ of order type $\alpha$.
        \item $A$ is \emph{$\alpha$-fickle} if $\alpha$ is the smallest ordinal such that $A$ is $\leq\alpha$-fickle.
        \item $\mathbf{d}$ is \emph{$\leq\alpha$-fickle} if every set in $\mathbf{d}$ is $\leq\alpha$-fickle.
        \item $\mathbf{d}$ is \emph{$\alpha$-fickle} if $\alpha$ is the smallest ordinal such that $\mathbf{d}$ is $\leq\alpha$-fickle.
    \end{enumerate}
\end{Definition}

Note that being $\alpha$-fickle was first introduced by \citeauthor{downey2020hierarchy} as being \emph{properly totally $\alpha$-computably approximable}. We use the term fickleness for succinctness. By canonicalness, if $\alpha$ is \emph{reasonably small}:
\[\alpha <\epsilon_0, \mathrm{where } \epsilon_0=\sup\{\omega, \omega^\omega, \omega^{\omega^\omega}, \ldots\},\]
then in Definition~\ref{def:fickleness}, the choice of $\mathcal{R}$ would not matter~\citep{downey2020hierarchy}, because we would be able to switch from one computable approximation to another recursively. Restricting to smalls ordinals is enough for now because the embeddability results discussed do not involve ordinals exceeding $\omega^\omega<\epsilon_0$. Also, note that when a set $A$ is $\alpha$-fickle for some $\alpha\geq\omega$, the ``number of mind changes'' of the set is actually still finite, in the sense that $\lim_s A(x,s)$ exists for all $x\in\omega$; nonetheless the smallest canonical computable approximation for $A$ has an infinite order type.

The fickleness hierarchy collapses to the powers of $\omega$:
\begin{Theorem}[\cite{downey2020hierarchy}] \label{thm:collapse}
    Let $\alpha\leq \epsilon_0$. $\alpha$-fickle degrees exist if and only if $\alpha$ is a power of $\omega$. Also, an \re degree is 1-fickle if and only if the degree is the zero degree.
\end{Theorem}

Recall that the \re degrees that bound $N_5$ ($L_7$, $M_3$) are exactly those that contain sets of fickleness $>1$ ($>\omega$, $\geq\omega^\omega$ respectively). From Theorem~\ref{thm:collapse}, the first non-trivial level that has not been characterized is at $\omega^2$, which makes our search for a $>\omega^2$-lattice meaningful.

\subsection{Some ``Short'' and ``Narrow'' Lattices} \label{sec:3direct}
We consider $>\omega^2$-candidates that are as ``short'' and ``narrow'' as $L_7$ and $M_3$. We introduce the notion of \emph{3-directness} to describe such smallish lattices:

\begin{Definition} \label{def:direct}
    A lattice $L$ is \emph{$n$-direct} if $L$ contains $n$ incomparable elements $\bm{a_0},\ldots,\bm{a_{n-1}}$, and every element in $L$ is of the form $\bigcup_{i\in I} \bm{a_i}$ or $\bigcap_{i\in I} \bm{a_i}$, where $I\subseteq \{0,\ldots,n-1\}$.
\end{Definition}

For instance, $L_7$ is 3-direct because we can name the middle element $\bm{a_1}$ and the two at the side $\bm{a_0}$ and $\bm{a_2}$. Then the top element will be $\bm{a_0}\cup\bm{a_1} =\bm{a_0}\cup\bm{a_2} =\bm{a_1}\cup\bm{a_2}$, the lower two elements will be $\bm{a_0}\cap\bm{a_1}$ and $\bm{a_1}\cap\bm{a_2}$, and the bottom element will be $\bm{a_0}\cap\bm{a_2} =\bm{a_0}\cap\bm{a_1}\cap\bm{a_2}$. Figure~\ref{fig:lempp-lerman} gives two lattices that are not 3-direct, for they each contain more than three incomparable elements.

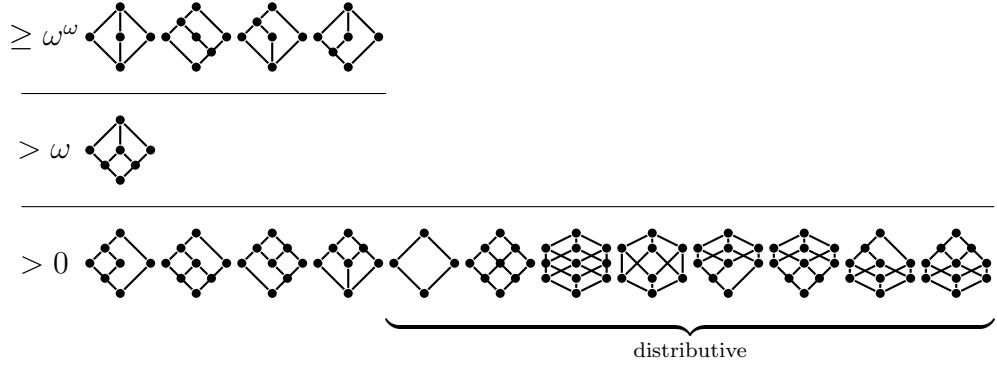
\begin{figure}[tpb]
    \centering
    \tikzset{every picture/.style={thick}}
    \begin{tikzpicture}[every node/.style={circle,fill=black,inner sep=1.2pt}]
        \tikzset{131/.pic ={
            \node (a) at (0, -1) {};
            \node (b) at (-1, 0) {};
            \node (c) at (0, 1) {};
            \node (d) at (1, 0) {};
            \node (e) at (0, 0) {};
            \draw (a) -- (b) -- (c) -- (d) -- (a) -- (e) -- (c);
        }}
        \tikzset{oo1/.pic ={
            \node (d) at (0,-1) {};
            \node (l) at (-1,0) {};
            \node (u) at (0,1) {};
            \node (r) at (1,0) {};
            \node (c) at (0,0) {};
            \node (ul) at (-0.5,0.5) {};
            \node (dr) at (0.5,-0.5) {};
            \draw (dr) --(ul) --(u) --(r) --(d) --(l) --(ul);
        }}
        \tikzset{oo2/.pic ={
            \node (d) at (0,-1) {};
            \node (l) at (-1,0) {};
            \node (u) at (0,1) {};
            \node (r) at (1,0) {};
            \node (c) at (0,0) {};
            \node (ul) at (-0.5,0.5) {};
            \draw (ul) --(u) --(r) --(d) --(l) --(ul) --(c) --(d);
        }}
        \tikzset{oo3/.pic ={
            \node (d) at (0,-1) {};
            \node (l) at (-1,0) {};
            \node (u) at (0,1) {};
            \node (r) at (1,0) {};
            \node (c) at (0,0) {};
            \node (dl) at (-0.5,-0.5) {};
            \draw (dl) --(c) --(u) --(r) --(d) --(l) --(u);
        }}
        \tikzset{l7/.pic ={
            \node (a) at (0, -1) {};
            \node (b) at (-1, 0) {};
            \node (c) at (0, 1) {};
            \node (d) at (1, 0) {};
            \node (e) at (0, 0) {};
            \draw (a)--(b)--(c)--(d)--(a);
            \node (f) at (-0.5, -0.5) {};
            \node (g) at (0.5, -0.5) {};
            \draw (c)--(e)--(f) (e)--(g);
        }}
        \tikzset{a0/.pic ={
            \node (d) at (0,-1) {};
            \node (l) at (-1,0) {};
            \node (u) at (0,1) {};
            \node (r) at (1,0) {};
            \node (c) at (0,0) {};
            \node (ul) at (-0.5,0.5) {};
            \node (dl) at (-0.5,-0.5) {};
            \draw (ul) --(u) --(r) --(d) --(l) --(ul) --(c) --(dl);
        }}
        \tikzset{a1/.pic ={
            \node (d) at (0,-1) {};
            \node (l) at (-1,0) {};
            \node (u) at (0,1) {};
            \node (r) at (1,0) {};
            \node (c) at (0,0) {};
            \node (ul) at (-0.5,0.5) {};
            \node (dl) at (-0.5,-0.5) {};
            \node (dr) at (0.5,-0.5) {};
            \draw (ul) --(u) --(r) --(d) --(l) --(ul) --(c) --(dl);
            \draw (c) --(dr);
        }}
        \tikzset{a2/.pic ={
            \node (d) at (0,-1) {};
            \node (l) at (-1,0) {};
            \node (u) at (0,1) {};
            \node (r) at (1,0) {};
            \node (c) at (0,0) {};
            \node (ul) at (-0.5,0.5) {};
            \node (ur) at (0.5,0.5) {};
            \node (dr) at (0.5,-0.5) {};
            \draw (ul) --(u) --(r) --(d) --(l) --(ul) --(c) --(dr);
            \draw (c)--(ur);
        }}
        \tikzset{a3/.pic ={
            \node (d) at (0,-1) {};
            \node (l) at (-1,0) {};
            \node (u) at (0,1) {};
            \node (r) at (1,0) {};
            \node (c) at (0,0) {};
            \node (ul) at (-0.5,0.5) {};
            \node (ur) at (0.5,0.5) {};
            \draw (ul) --(u) --(r) --(d) --(l) --(ul) --(c) --(d);
            \draw (c) --(ur);
        }}
        \tikzset{a4/.pic ={
            \node (d) at (0,-1) {};
            \node (l) at (-1,0) {};
            \node (u) at (0,1) {};
            \node (r) at (1,0) {};
            \node (c) at (0,0) {};
            \node (ul) at (-0.5,0.5) {};
            \node (ur) at (0.5,0.5) {};
            \node (dl) at (-0.5,-0.5) {};
            \node (dr) at (0.5,-0.5) {};
            \draw (ul) --(u) --(r) --(d) --(l) --(ul) --(dr);
            \draw (ur) --(dl);
        }}
        \tikzset{b0/.pic ={
            \node (1) at (0,1) {};
            \node (0) at (0,-1) {};
            \node (c) at (0,0) {};
            \node (cl) at (-1,0) {};
            \node (cr) at (1,0) {};
            \node (u) at (0,0.5) {};
            \node (ul) at (-1,0.5) {};
            \node (ur) at (1,0.5) {};
            \node (d) at (0,-0.5) {};
            \node (dl) at (-1,-0.5) {};
            \node (dr) at (1,-0.5) {};
            \draw (1) --(ur) --(c) --(ul) --(1);
            \draw (u) --(cr) --(d) --(cl) --(u);
            \draw (c) --(dr) --(0) --(dl) --(c);
            \draw (ul)--(dl) (ur)--(dr) (1)--(u) (0)--(d);
        }}
        \tikzset{b1/.pic ={
            \node (1) at (0,1) {};
            \node (0) at (0,-1) {};
            \node (u) at (0,0.5) {};
            \node (ul) at (-1,0.5) {};
            \node (ur) at (1,0.5) {};
            \node (d) at (0,-0.5) {};
            \node (dl) at (-1,-0.5) {};
            \node (dr) at (1,-0.5) {};
            \draw (1) --(ur) --(d) --(ul) --(1);
            \draw (u) --(dr) --(0) --(dl) --(u);
            \draw (ul) --(dl) (ur)--(dr) (1)--(u) (0)--(d);
        }}
        \tikzset{b2/.pic ={
            \node (1) at (0,1) {};
            \node (0) at (0,-1) {};
            \node (c) at (0,0) {};
            \node (cl) at (-1,0) {};
            \node (cr) at (1,0) {};
            \node (u) at (0,0.5) {};
            \node (ul) at (-1,0.5) {};
            \node (ur) at (1,0.5) {};
            \node (dl) at (-0.5,-0.5) {};
            \draw (1) --(ur) --(c) --(ul) --(1);
            \draw (u) --(cr) --(0) --(cl) --(u);
            \draw (ul) --(cl) (ur)--(cr) (1)--(u) (c)--(dl);
        }}
        \tikzset{b3/.pic ={
            \node (1) at (0,1) {};
            \node (0) at (0,-1) {};
            \node (c) at (0,0) {};
            \node (cl) at (-1,0) {};
            \node (cr) at (1,0) {};
            \node (u) at (0,0.5) {};
            \node (ul) at (-1,0.5) {};
            \node (ur) at (1,0.5) {};
            \node (dl) at (-0.5,-0.5) {};
            \node (dr) at (0.5,-0.5) {};
            \draw (1) --(ur) --(c) --(ul) --(1);
            \draw (u) --(cr) --(0) --(cl) --(u);
            \draw (ul) --(cl) (ur)--(cr) (1)--(u) (dl)--(c) --(dr);
        }}
        \tikzset{b4/.pic ={
            \node (1) at (0,1) {};
            \node (0) at (0,-1) {};
            \node (c) at (0,0) {};
            \node (cl) at (-1,0) {};
            \node (cr) at (1,0) {};
            \node (ul) at (-0.5,0.5) {};
            \node (d) at (0,-0.5) {};
            \node (dl) at (-1,-0.5) {};
            \node (dr) at (1,-0.5) {};
            \draw (1) --(cr) --(d) --(cl) --(1);
            \draw (c) --(dr) --(0) --(dl) --(c);
            \draw (cl) --(dl) (cr)--(dr) (0)--(d) (c)--(ul);
        }}
        \tikzset{b5/.pic ={
            \node (1) at (0,1) {};
            \node (0) at (0,-1) {};
            \node (c) at (0,0) {};
            \node (cl) at (-1,0) {};
            \node (cr) at (1,0) {};
            \node (ul) at (-0.5,0.5) {};
            \node (ur) at (0.5,0.5) {};
            \node (d) at (0,-0.5) {};
            \node (dl) at (-1,-0.5) {};
            \node (dr) at (1,-0.5) {};
            \draw (1) --(cr) --(d) --(cl) --(1);
            \draw (c) --(dr) --(0) --(dl) --(c);
            \draw (cl) --(dl) (cr)--(dr) (0)--(d) (ul)--(c)--(ur);
        }}
        \tikzset{diamond/.pic ={
            \node (l) at (-1,0) {};
            \node (r) at (1,0) {};
            \node (u) at (0,1) {};
            \node (d) at (0,-1) {};
            \draw (u)--(r)--(d)--(l)--(u);
        }}
        \tikzset{goo/.pic ={
            \node[draw=none,fill=none] at (-2,0) {$\geq\omega^\omega$};
            \path (-1,0) pic[scale=0.4]{131};
            \path (0,0) pic[scale=0.4]{oo1};
            \path (1,0) pic[scale=0.4]{oo2};
            \path (2,0) pic[scale=0.4]{oo3};
        }}
        \tikzset{go/.pic ={
            \node[draw=none,fill=none] at (-0.5,0) {$>\omega$};
            \path (0.5,0) pic[scale=0.4]{l7};
        }}
        \tikzset{g0/.pic ={
            \node[draw=none,fill=none] at (-5,2) {$>0$};
            \path (-4,2) pic[scale=0.4]{a0};
            \path (-3,2) pic[scale=0.4]{a1};
            \path (-2,2) pic[scale=0.4]{a2};
            \path (-1,2) pic[scale=0.4]{a3};
            \path ( 0,2) pic[scale=0.4]{diamond};
            \path ( 1,2) pic[scale=0.4]{a4};
            \path ( 2,2) pic[scale=0.4]{b0};
            \path ( 3,2) pic[scale=0.4]{b1};
            \path ( 4,2) pic[scale=0.4]{b2};
            \path ( 5,2) pic[scale=0.4]{b3};
            \path ( 6,2) pic[scale=0.4]{b4};
            \path ( 7,2) pic[scale=0.4]{b5};
        }}
        \path (0,1.5) pic[align=left]{goo};
        \draw [thin] (-2.3,0.75) --(2.5,0.75);
        \path (-1.5,0) pic[align=left]{go};
        \draw [thin] (-2.3,-0.75) --(10.5,-0.75);
        \path (3,-3.5) pic[align=left]{g0};
        \node[rectangle,draw=none,fill=none] at (6.5,-2.5) {$\underbrace{\hspace{8cm}}_{\mathrm{distributive}}$};
    \end{tikzpicture}
    \caption{All $\leq3$-direct lattices (Definition~\ref{def:direct}) are exhausted here, grouped by the fickleness levels characterized. We include also the diamond, which is the only $\leq2$-direct lattice, if we ignore trivial lattices such as the single point. Lattices including and to the right of the diamond are distributive.}
    \label{fig:3direct}
\end{figure}

In Figure~\ref{fig:3direct}, we exhibit some $\leq3$-direct lattices, which turns out to exhaust all of them~\citep{ko2021thesis}. Eight of them are distributive and can therefore be embedded below any nonzero \re degree~\citep{thomason1971sublattices}. Excluding $L_7$ and $M_3$, we are left with seven non-distributive lattices to characterize in the next two sections.

\section{Some Lattices where $>1$ or $\geq\omega^\omega$-Fickleness is Sufficient} \label{sec:sufficient}
We can apply the techniques in \citep{downey2020hierarchy} to show:
\begin{Lemma} \label{lemma:3direct-oo}
    The four lattices at the top row of Figure~\ref{fig:3direct}, which includes $M_3$, can be bounded below any \re degree that contains $\geq\omega^\omega$-fickle sets.
\end{Lemma}

Known methods can also show:
\begin{Lemma} \label{lemma:3direct-2}
    The first three lattices in the bottom row of Figure~\ref{fig:3direct} can be bounded below any non-zero \re degree.
\end{Lemma}

Lemma~\ref{lemma:3direct-2} is an application of the methods in \citep{lerman1984elementary}. The authors introduced a lattice-theoretic property known as the \emph{Trace-Probe Property} (TPP), which is sufficient for a lattice to be embeddable in $\Rt$. It is not hard to show that their TPP embedding technique is compatible with nonzero permitting. The three lattices satisfy TPP, and can therefore be bounded below any nonzero \re degree. The fourth lattice in the bottom row, enlarged in Figure~\ref{fig:3direct-n}, does not satisfy TPP. But we can still prove:

\begin{Lemma} \label{lemma:3direct-n}
    The lattice in Figure~\ref{fig:3direct-n} can be bounded below any non-zero \re degree.
\end{Lemma}

In the rest of this section, we shall prove the lemma by modifying the TPP techniques, and applying nonzero permitting. The fickleness requested by a lattice comes from its positive requirements $\rho$. If $\rho$ requests for $n$ many permissions, then any set that can grant $\geq n$ permissions can satisfy $\rho$. In TPP lattices, $n\in\omega$ is a constant, while in the lattice of Lemma~\ref{lemma:3direct-n}, $n$ is bounded by $|\rho|\in\omega$. Since nonzero \re sets can grant $n$ many permissions for any fixed $n$, both types of lattices can be bounded below any nonzero degree.

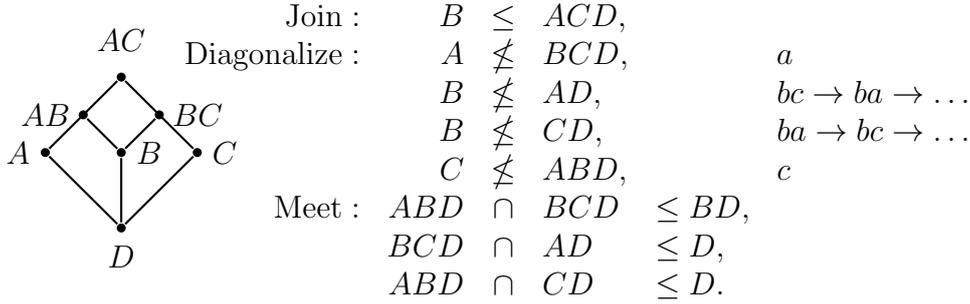
\begin{figure}[tpb]
    \centering
    \tikzset{every picture/.style={thick}}
    \begin{tikzpicture}[every node/.style={circle,fill=black,inner sep=1.2pt}]
        \tikzset{pics/a3/.style n args={2}{code ={
            \node[label=below:{$D$}] (d) at (0,-1) {};
            \node[label=left:{$A$}] (l) at (-1,0) {};
            \node[label={[label distance=0em]above:{$AC$}}] (u) at (0,1) {};
            \node[label=right:{$C$}] (r) at (1,0) {};
            \node[label=right:{$B$}] (c) at (0,0) {};
            \node[label=left:{$AB$}] (ul) at (-0.5,0.5) {};
            \node[label=right:{$BC$}] (ur) at (0.5,0.5) {};
            \draw (c)--(d)--(l)--(u)--(r)--(d) (ul)--(c)--(ur);
            }}, pics/a3/.default={x}{y}
        }
        \tikzset{pics/eq3/.style n args={2}{code ={
            \node[rectangle,draw=none,fill=none] at (0,0) {$\begin{array}{rrclll}
                \mathrm{Join: } &B &\leq &ACD,\\
                \mathrm{Diagonalize: } &A &\nleq &BCD, &&a\\
                &B &\nleq &AD, &&bc\to ba \to \ldots\\
                &B &\nleq &CD, &&ba\to bc \to \ldots\\
                &C &\nleq &ABD, &&c\\
                \mathrm{Meet: } &ABD &\cap &BCD &\leq BD,\\
                &BCD &\cap &AD &\leq D,\\
                &ABD &\cap &CD &\leq D.
            \end{array}$};
            }}, pics/eq3/.default={x}{y}
        }
        \path (0,0) pic{a3};
        \path (6,0) pic{eq3};
    \end{tikzpicture}
    \caption{Here is a lattice that can be bounded below any non-zero \re degree. $\rho:B\nleq AD$ generates traces of length two that retarget between $BC$ and $BA$-traces, where the number of times we re-target is bounded by $|\rho|$. In particular, by the notation of \citeauthor{lerman1984elementary} (\citeyear{lerman1984elementary}), this lattice is not a TPP lattice because the bound is not independent from $\rho$.}
    \label{fig:3direct-n}
\end{figure}

To see how $|\rho|$ comes about, we need to observe the interactions between $\rho$ and other requirements, such as with negative requirements $\eta$, and with joins $J$. We borrow notations from \cite{downey2007totally,downey2020hierarchy,lempp2012priority,ambos2019universally,lempp2006embedding}. In these constructions, $\rho$ generates sequences of elements that need to be enumerated. The sequences are referred to as \emph{traces}, and each element in a trace is often referred to as a \emph{ball}. The process of enumerating can be thought of as balls falling to the bottom of a pinball machine, into buckets that represent each ball's targeted set. Often, longer traces are \emph{partitioned} into shorter ones, so that when the balls in a trace are enumerated simultaneously, $\eta$ will not be injured in an unforeseeable manner. A trace can be as short as a single ball in length. Each trace enumeration must be permitted by the set bounding the lattice. A permission is granted when the set changes its mind, or in other words, when the set loses one fickleness. Henceforth, we shall use the terms ``permissions'' and ``fickleness'' interchangeably. Therefore, the length of the traces generated by $\rho$ bounds the fickleness sufficient for embedding $L$. As we shall see, these lengths do not exceed $|\rho|$.

Let $X$ be an \re set of nonzero degree. We want to construct \re sets $A,B,C,D\leq X$ satisfying the requirements in Figure~\ref{fig:3direct-n}. To derive these requirements, we use $AB$ to abbreviate $A$ join $B$. We also let the \re set that represents a point on the lattice be the join of its label, with the sets of labels below it. For instance, the point labelled $A$ will be represented by $AD$, the join of sets $A$ and $D$, while the point labelled $AC$ will be represented by $ABCD$. This rule preserves the partial orders of the lattice, allowing us to omit requirements such as $D\leq A$ for instance, since $D\leq AD$. Because of this rule, to express ``$B\leq AC$'', which says that the join of the points $A$ and $C$ computes $B$, we use the requirement $B\leq ACD$. Also, note that we omitted requirements that can be derived from those listed in the figure. For instance, $BD\nleq AD$ holds if and only if $B\nleq AD$ holds, so we omit the longer version.

Fix a computable ordering of the requirements. We shall use a $\Pi_2$ tree framework in our construction. First, consider the global requirements $A,B,C,D\leq X$, where we construct functionals $\Psi_A,\Psi_B,\Psi_C,\Psi_D$ to satisfy:
\begin{align*}
    G_A: &\;A = \Psi_A(X),\\
    G_B: &\;B = \Psi_B(X),\\
    G_C: &\;C = \Psi_C(X),\\
    G_D: &\;D = \Psi_D(X).
\end{align*}
We describe the strategy for $G_A$ only; the remaining $G$-strategies are similar.

\begin{framed}
    \noindent\textbf{$G_A$-strategy}: \emph{Initialize} $G_A$ by setting $\Psi=\Psi_A=\emptyset$. Start with $a=0$.
    \begin{enumerate}
        \item \emph{Set up}: At stage $a$, some positive requirement $\rho$ might pick $a$ as a follower and assign a \emph{code} $i$ to $a$. Set use $\psi(a)=i$. At the next stage start a new cycle with $a+1$.
        
        \item \emph{Maintain}: Finitely often, $\rho$ might want to lift the use $\psi(a)$ or even enumerate $a$ into $A$. \emph{Permission is granted by $X$} to take those actions at the stage when $X\restriction \psi(a)$ changes. With permission, $\rho$ may also change the code of $a$ to $i'$, and we update the use to $\psi(a)=i'$.
        
        Note that it is possible for many elements to be assigned the same code. These elements wait for the same permission from $X$ to be enumerated, and if they come from the same follower of a positive requirement, they will be enumerated together.
    \end{enumerate}
\end{framed}

Next, consider the join requirement, which is also global, where we construct functional $K=K_B$ such that
\begin{align*}
    J_B: B &=K_B(ACD).
\end{align*}

Since this lattice has no other join requirements, we drop the subscript $B$.
\begin{framed}
    \noindent\textbf{$J$-strategy}: \emph{Initialize} $J$ by setting $K=\emptyset$. Start with $n=0$.
    \begin{enumerate}
        \item \emph{Set up}: At some stage $\geq n$, pick a large value as use $\kappa(n)$ targeted for $A$, $C$, or $D$, and set $K(n)=0$ with that use. The choice of the target will depend on the positive requirement $\rho$, if any, with which $n$ is associated. Start a new cycle with $n+1$. Note that $n$ is targeted for $B$, and as mentioned in the $G$-strategies, $n$ will be assigned by other requirements a code $i$. Assign the same code to $\kappa(n)$.
        
        \item \emph{Maintain}: We may \emph{lift the use at $n$} by enumerating $\kappa(n)$ into $A$, $C$, or $D$, then choosing a new large $\kappa(n)$. With the new use, we declare $K(n)=0$ if $n$ has been enumerated into $B$, otherwise we declare $K(n)=1$. The new use will be targeted for one of $A$, $C$, or $D$ again. Depending on $\rho$, the choice of the target may change, and we refer to this change as \emph{re-targeting}. Whenever we lift the use, the requirement that wanted the lifting may assign to the new use a new code. We will need to ensure that $\rho$ does not lift $\kappa(n)$ infinitely often.
    \end{enumerate}
\end{framed}

Consider the meet requirements $\eta$, where we construct $\Theta$ such that:
\begin{align*}
    \eta_{\Phi W}: &\Phi_0(ABD) =\Phi_1(BCD) =W \implies (\exists \Theta) [W=\Theta(BD)],\\
    \eta_{BA\Phi W}: &\Phi_0(BCD) =\Phi_1(AD) =W \implies (\exists \Theta) [W=\Theta(D)],\\
    \eta_{BC\Phi W}: &\Phi_0(ABD) =\Phi_1(CD) =W \implies (\exists \Theta) [W=\Theta(D)].
\end{align*}

We drop subscripts if context is clear. Consider $\eta_{BA}$, also referred to as \emph{$BA$-gates}, because $B$ and $A$ appear on opposite sides of the gate, which as we shall see, generally forbids $B$- and $A$-balls from being enumerated at the same time. Similarly, $\eta_{BC}$ is often referred to as a $BC$-gate. Maintaining a computation $W(m)=\Theta(m)$ at subrequirement $\eta_{AB},m$ involves protecting the computation at the $BCD$-side $\Phi_0(BCD,m)$, or the $AD$-side $\Phi_1(AD,m)$. In the $BA$-gate, if the hypothesis $\Phi_0=\Phi_1=W$ holds, we want to ensure that $W\leq D$, and hence refer to $D$ as a \emph{computing set(s)} of the gate. Similarly, the computing sets of the two other types of gates are $BD$ and $D$.

A gate's strategy is based on the minimal pair strategy of \citep{yates1966minimal}, where whenever both sides of a computation are simultaneously injured, say at subrequirement $m$, an element $\leq\theta(m)$ must enter a computing set, which we refer to as the use $\theta(m)$ being \emph{lifted}:
\begin{framed}
    \noindent \textbf{$\eta$-strategy}: \emph{Initialize} $\eta$ by setting $\Theta=\emptyset$ and $m=0$ and \emph{use} $\theta(m')$ to be undefined for all $m'\in\omega$.
    \begin{enumerate}
        \item \emph{Set up}: Wait for \emph{equality} 
        \[\Phi_0\restriction (m+1) =\Phi_1\restriction (m+1) =W\restriction (m+1)\]
        at $m$ for the first time, say at stage $s_0$. Start a new phase at set up for the $(m+1)$-th subrequirement.
        
        \item \emph{Maintain}: When equality is first observed or has returned and $\theta(m)$ does not exist, assign a new large value for $\theta(m)$. Whenever $\Phi_0(m)$ or $\Phi_1(m)$ is injured by a follower that existed at stage $s_0$, and an element $\leq\theta(m)$ enters a computing set at the same time, we lift the use by declaring $\theta(m)$ to be undefined. Then we wait for the next $\eta$-stage before returning to the beginning of this maintenance phase to pick a new large use.
    \end{enumerate}
\end{framed}
        
To verify that $\eta$ is satisfied, we would like to show that the only elements that can injure $\Phi_0(m)$ or $\Phi_1(m)$ are those that existed at stage $s_0$. Also, $\Phi_0(m)$ and $\Phi_1(m)$ cannot be injured together unless $\theta(m)$ is lifted. Note that if equality at $m$ never returns, $\eta$ would be satisfied vacuously. If the hypothesis of the gate holds, we will need to show that the use is not lifted infinitely often. We shall provide such a verification in Lemma~\ref{lemma:eta} for a more complex construction.

Now, consider the diagonalization or positive requirements, say $B\not\leq AD$. Given functional $\Xi$, we satisfy
\[\rho_{\Xi}: B\neq \Xi(AD).\]

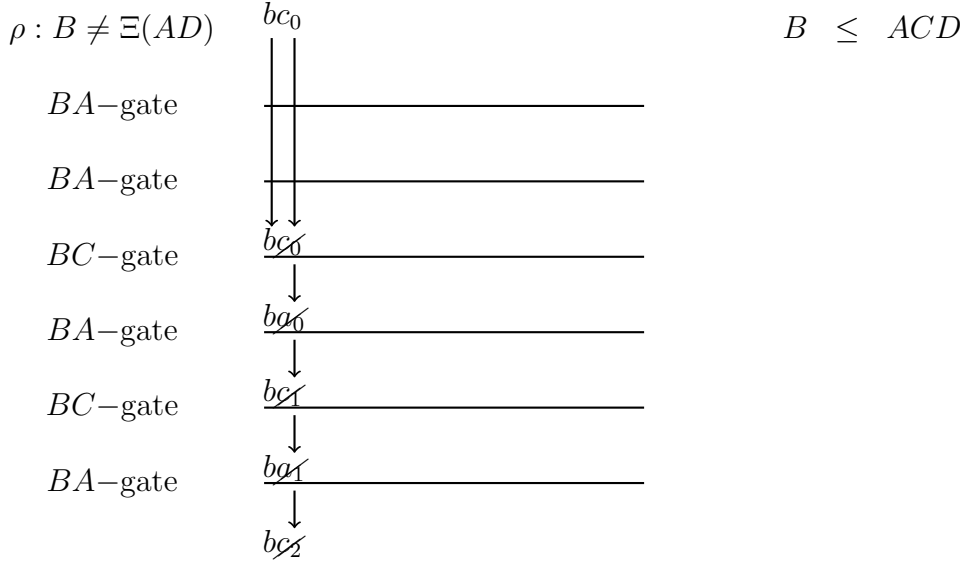
\begin{figure}[tpb]
    \centering
    \tikzset{every picture/.style={thick}}
    \begin{tikzpicture}[every node/.style={rectangle,draw=none,fill=none}]
        \tikzset{pics/rho/.style n args={3}{code ={
            \node at (-2,0) {$#1$};
            \node[text width=10em,align=left] at (2.05+#2,0.2) {$#3$};
            }}, pics/rho/.default={\eta}{0}{abab}
        }
        \tikzset{pics/gate/.style n args={3}{code ={
            \node at (-2,0) {$#1$};
            \node[text width=10em,align=left] at (2.05+#2,0.2) {$#3$};
            \draw (0,0)--(5,0);
            }}, pics/gate/.default={\eta}{0}{abab}
        }
        \tikzset{pics/down/.style n args={3}{code ={
            \draw[->] (0.4+#1,#2-0.1)--(0.4+#1,#2-0.6-#3);
            }}, pics/down/.default={0}{2}{0}
        }
        \tikzset{traces/.pic ={
            \path (0,1) pic{rho={\rho:B\neq \Xi(AD)}{0}{bc_0}};
            \path (0,0) pic{down={0.0}{1}{2}};
            \path (0,0) pic{down={-0.3}{1}{2}};
            \path (0, 0) pic{gate={BA\mathrm{-gate}}{0}{ }};
            \path (0,-1) pic{gate={BA\mathrm{-gate}}{0}{ }};
            \path (0,-2) pic{gate={BC\mathrm{-gate}}{0}{b\cancel{c_0}}};
            \path (0,0) pic{down={0.0}{-2}{0}};
            \path (0,-3) pic{gate={BA\mathrm{-gate}}{0}{b\cancel{a_0}}};
            \path (0,0) pic{down={0.0}{-3}{0}};
            \path (0,-4) pic{gate={BC\mathrm{-gate}}{0}{b\cancel{c_1}}};
            \path (0,0) pic{down={0.0}{-4}{0}};
            \path (0,-5) pic{gate={BA\mathrm{-gate}}{0}{b\cancel{a_1}}};
            \path (0,0) pic{down={0.0}{-5}{0}};
            \path (0,-6) pic{rho={}{0}{b\cancel{c_2}}};
        }}
        \tikzset{pics/eq/.style n args={2}{code ={
            \node[rectangle,draw=none,fill=none] at (0,0) {$\begin{array}{rclll}
                B &\leq &ACD
            \end{array}$};
            }}, pics/eq/.default={x}{y}
        }
        \path (0,0) pic{traces};
        \path (8,1) pic{eq};
    \end{tikzpicture}
    \caption{Traces generated by $\rho$, after interacting with the join on the right, and with gates depicted by a horizontal line. Gates of higher priority are drawn lower down. Whenever the trace meets consecutive $BA$- and $BC$-gates, it is re-targeted between $BC$- and $BA$-traces. Therefore, the total fickleness requested does not exceed $|\rho|$. In this toy example, the length of a stopped trace does not exceed 2. In more complex constructions, traces can be arbitrarily long, like those in Figure~\ref{fig:traces-layer}.}
    \label{fig:traces-3direct-n}
\end{figure}

The strategy is built on the usual Friedberg Muchnik approach, where we pick a follower $b$, wait for $b$ to be \emph{realized} when $\Xi(b)=0$, set a restraint of $\xi(b)$ on $AD$ to prevent un-realization, then wait for permission from $X$ to enumerate $b$ into $B$. But the joins and gates complicate this enumeration. Consider how $\rho=\rho_{\Xi}$ interacts with $J$ and with higher priority $BA$-gates and $BC$-gates. Refer to Figure~\ref{fig:traces-3direct-n} for an illustration. Assume $\rho$ has a follower $b$ that is realized at stage $s_0$. We want to enumerate $b$ into $B$ at the same time prevent elements $\leq\xi(b)$ from entering $AD$. But $J$ says that before we can enumerate $b$, we need to enumerate $\kappa(b)$ into targets $A$, $C$, or $D$. If $b$ was realized after $J$ chose $\kappa(b)$, then $\xi(b)>\kappa(b)$ and we cannot choose $A$ or $D$ as targets otherwise we will un-realize $b$. So we can only target $C$. So we want to enumerate $c_0=\kappa(b)$ into $C$, then $b$ into $B$. We think of these elements $c_0$ and $b$ as \emph{balls} falling past gates of a pinball machine, to be collected by their target sets $C$ and $B$ respectively. The enumeration of the $B$-ball demanded that a $C$-ball be enumerated first, and we say that the $B$-ball generates an \emph{$BC$-trace}.

One might be tempted to enumerate the $BC$-trace simultaneously. But $BC$-gates prevent us from doing so, unless the gate's use $d$ targeted for $D$ is lifted at the same time. But if $d$ was picked before $b$ got realized, then $d<\xi(b)$ cannot be lifted without un-realizing $b$. Thus the $BC$-gate forces us to enumerate the balls in the $BC$-trace separately and in reverse, starting from the $C$-ball, say at stage $s_1\geq s_0$. We say that the trace is \emph{stopped} by the $BC$-gate from simultaneous enumeration. After enumerating $c_0$, we have to wait for the $C$-side computation of stopping gate to recover at some stage $s_2>s_1$, before we can enumerate the $B$-ball. When recovery has occurred, we say that the $BC$-gate has \emph{reopened}. While waiting for recovery, $J$ requires that we pick a new large use to be re-targeted to $A$, $C$, or $D$. The new target cannot be on the $C$-side of the $BC$-gate; otherwise, when the new use is enumerated later, we will injure the $C$-side again. Therefore we cannot re-target to $C$ or $D$, and are forced to choose $A$. Thus we re-target $bc_0$ to a $BA$-trace $ba_0$. In other words, once $c_0$ is enumerated, we need to pick an $A$-ball $a_0$ and wait for it to be enumerated into $A$, before we are allowed to enumerate $b$ into $B$. After the $C$-enumeration, the trace is no longer stopped by the $BC$-gate, and we say that it has \emph{passed} the gate.

After the $C$-side of the $BC$-gate recovers, one might again be tempted to enumerate the $BA$-trace $ba_0$ simultaneously when permitted. But higher priority $BA$-gates disallows $B$- and $A$-balls from being enumerated together. We say that the trace is \emph{stopped} at this $BA$-gate. Repeating the argument in the above paragraph but with the role of $A$ and $C$ swapped, we will want to enumerate $a_0$ first, then retarget the $BA$-trace to a $BC$-trace $bc_1$. Again if there is a higher priority $BC$-gate, we will need to retarget to a $BA$-trace $ba_1$, and so on.

Suppose $\rho$ has $k$ many higher priority gates
\[\eta_1 <\ldots <\eta_k <\rho,\]
that alternate between $BA$- and $BC$-gates. Refer to Figure~\ref{fig:traces-3direct-n} for an illustration, where higher priority gates are drawn lower down. The trace of $\rho$ first descends to the lowest priority gate $\eta_k$ and is stopped by that gate. After retargeting, enumerating, and waiting for the gate to reopen, the trace descends to the next lower priority gate $\eta_{k-1}$, where it again waits to pass. This process is repeated $k$ many times until the trace passes gate $\eta_0$, and the elements of the trace is enumerated simultaneously. Each enumeration requests one permission from $X$. Given that $k<|\rho|$, $|\rho|$ many fickleness will be sufficient to satisfy $\rho$.

The last type of gate $\eta$, which asserts that $ABD\cap BCD \leq BD$, will not stop $BC$ or $BA$-traces, because $\eta$'s computing sets includes $B$, and $B$-elements are picked only as followers $b$ and not as trace extensions, meaning that the picking occurs at $\eta$-expansionary stages. At these stages, if a subrequirement $\eta,m$ has already been set up, then the $b$ will be too large to inflict injury; and if the subrequirement has yet to be set up, then $\theta(m)$ can always be lifted above $b$, allowing for simultaneous injuries from $b$ on $\eta,m$. Details can be found in \citep{ko2021thesis}. Using the same framework, we can show that the other $\rho$ requirements generate traces no longer than $|\rho|$, which allows any nonzero \re degree to satisfy all requirements, and hence bound the lattice.

\section{Some Lattices where $\geq\omega^\omega$-Fickleness is Necessary} \label{sec:necessary}
In \citeyear{downey2020hierarchy}, when Downey and Greenberg showed that the $M_3$ needed $\geq\omega^\omega$-fickleness, they proved:
\begin{Theorem}[\cite{downey2020hierarchy}] \label{thm:oo-original}
    Let $L$ be a structure with 3 elements $A,B,C$ satisfying
    \begin{align*}
        \begin{matrix}
            A\leq BC,\\
            B\leq AC,\\
            C\leq AB,\\
            (\forall W \leq A, B)\; [W \leq C],\\
            (\forall W \leq A, C)\; [W \leq B],\\
            A\not\leq C\;\; (\mathrm{or } B\not\leq C \mathrm{ or } A\not\leq B).
        \end{matrix}
    \end{align*}
    If $\bm{d}$ is an \re degree that bounds $L$, $\bm{d}$ must contain a $\geq\omega^\omega$-fickle set.
\end{Theorem}

The conditions in the theorem are satisfied by only two lattices in Lemma~\ref{lemma:3direct-oo}, including $M_3$. To show that all lattices in the lemma demand $\geq\omega^\omega$-fickleness, we strengthen Theorem~\ref{thm:oo-original} by removing one join condition:
\begin{Theorem} \label{thm:oo}
    Let $L$ be a structure with 3 elements $A,B,C$ satisfying
    \begin{align*}
        \begin{matrix}
            A\leq BC,\\
            B\leq AC,\\
            (\forall W \leq A, B)\; [W \leq C],\\
            (\forall W \leq A, C)\; [W \leq B],\\
            A\not\leq C\;\; (\mathrm{or } B\not\leq C \mathrm{ or } A\not\leq B).
        \end{matrix}
    \end{align*}
    If $\bm{d}$ is an \re degree that bounds $L$, $\bm{d}$ must contain a $\geq\omega^\omega$-fickle set.
\end{Theorem}

All lattices in the lemma satisfy the conditions in the strengthened Theorem~\ref{thm:oo}, allowing us to apply the theorem with Lemma~\ref{lemma:3direct-oo} to get:

\begin{Corollary} \label{cor:oo}
    The four lattices in the top row of Figure~\ref{fig:3direct} characterize the \re degrees that contain $\geq\omega^\omega$-fickle sets.
\end{Corollary}

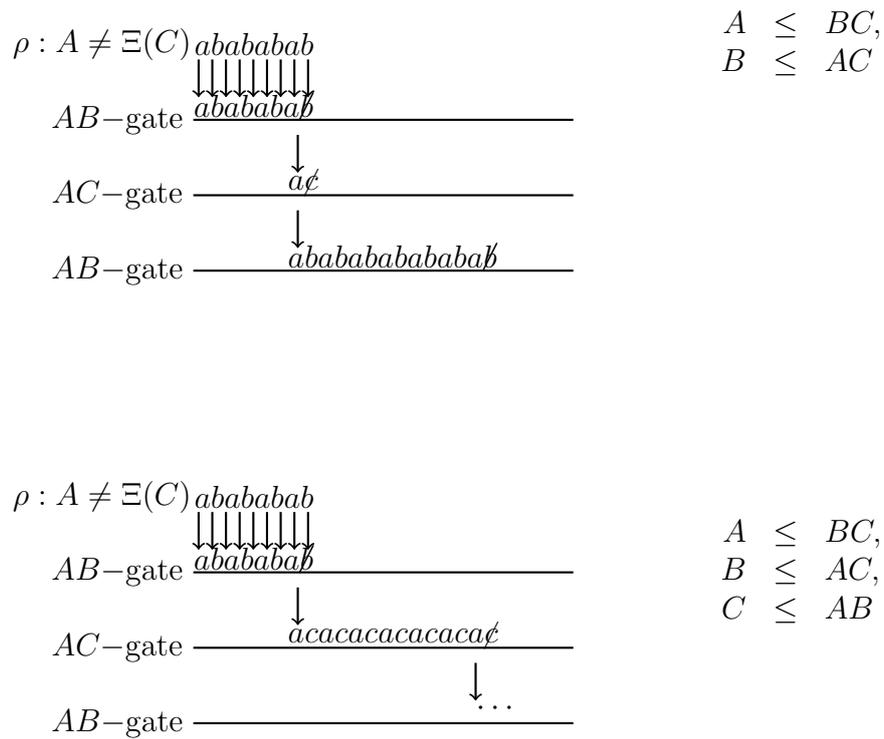
\begin{figure}[tpb]
    \centering
    \tikzset{every picture/.style={thick}}
    \begin{tikzpicture}[every node/.style={rectangle,draw=none,fill=none}]
        \tikzset{pics/gate/.style n args={3}{code ={
            \node at (-1,0) {$#1$};
            \node[text width=10em,align=left] at (59+#2em,0.2) {$#3$};
            \draw (0,0)--(5,0);
            }}, pics/gate/.default={\eta}{0}{abab}
        }
        \tikzset{pics/down/.style n args={2}{code ={
            \draw[->] (0.45*#1-0.38,#2-0.2)--(0.45*#1-0.38,#2-0.7);
            }}, pics/down/.default={0}{2}
        }
        \tikzset{traces1/.pic ={
            \node at (-1.2,1) {$\rho:A\neq \Xi(C)$};
            \node at (0.8,1) {$abababab$};
            \path (0,0) pic{down={1.0}{1}};
            \path (0,0) pic{down={1.4}{1}};
            \path (0,0) pic{down={1.8}{1}};
            \path (0,0) pic{down={2.2}{1}};
            \path (0,0) pic{down={2.6}{1}};
            \path (0,0) pic{down={3.0}{1}};
            \path (0,0) pic{down={3.4}{1}};
            \path (0,0) pic{down={3.8}{1}};
            \path (0,0) pic{down={4.2}{1}};
            \path (0, 0) pic{gate={AB\mathrm{-gate}}{0}{abababa\cancel{b}}};
            \path (0,-1) pic{gate={AC\mathrm{-gate}}{3}{a\cancel{c}}};
            \path (0,-2) pic{gate={AB\mathrm{-gate}}{3}{ababababababa\cancel{b}}};
            \path (0,0) pic{down={3.9}{0}};
            \path (0,0) pic{down={3.9}{-1}};
        }}
        \tikzset{traces2/.pic ={
            \node at (-1.2,1) {$\rho:A\neq \Xi(C)$};
            \node at (0.8,1) {$abababab$};
            \path (0,0) pic{down={1.0}{1}};
            \path (0,0) pic{down={1.4}{1}};
            \path (0,0) pic{down={1.8}{1}};
            \path (0,0) pic{down={2.2}{1}};
            \path (0,0) pic{down={2.6}{1}};
            \path (0,0) pic{down={3.0}{1}};
            \path (0,0) pic{down={3.4}{1}};
            \path (0,0) pic{down={3.8}{1}};
            \path (0,0) pic{down={4.2}{1}};
            \path (0, 0) pic{gate={AB\mathrm{-gate}}{0}{abababa\cancel{b}}};
            \path (0,0) pic{down={3.9}{0}};
            \path (0,-1) pic{gate={AC\mathrm{-gate}}{3}{acacacacacaca\cancel{c}}};
            \path (0,0) pic{down={9.1}{-1}};
            \path (0,-2) pic{gate={AB\mathrm{-gate}}{9}{\cdots}};
        }}
        \tikzset{pics/eq1/.style n args={2}{code ={
            \node[rectangle,draw=none,fill=none] at (0,0) {$\begin{array}{rclll}
                A &\leq &BC,\\
                B &\leq &AC
            \end{array}$};
            }}, pics/eq1/.default={x}{y}
        }
        \tikzset{pics/eq2/.style n args={2}{code ={
            \node[rectangle,draw=none,fill=none] at (0,0) {$\begin{array}{rclll}
                A &\leq &BC,\\
                B &\leq &AC,\\
                C &\leq &AB
            \end{array}$};
            }}, pics/eq2/.default={x}{y}
        }
        \path (0,0) pic{traces1};
        \path (8,1) pic{eq1};
        \path (0,-6) pic{traces2};
        \path (8,-6) pic{eq2};
    \end{tikzpicture}
    \caption{The top set of traces is obtained from the conditions in Theorem~\ref{thm:oo}, while the bottom set is from Theorem~\ref{thm:oo-original}, which has the additional join $J_C:C\leq AB$ shown at the right. The traces are generated by $\rho$. Without $J_C$, the top traces do not grow at the $AC$-gates, though they can still grow at the $AB$-gates. The lengths of the traces are bounded by $\omega^{|\rho|/2}$ and $\omega^{|\rho|}$ respectively.}
    \label{fig:traces-layer}
\end{figure}

The rest of this section will be used to prove Theorem~\ref{thm:oo}, where we tighten the layering technique of Theorem~\ref{thm:oo-original}. To get $\omega^\omega$, we generalize the framework in Section~\ref{sec:sufficient}, where we described the rules for the \emph{extending} and \emph{partitioning} of traces, so as to appease higher priority gates.

In that framework, diagonalization conditions, such as $\rho: A\neq \Xi(C)$ in Theorem~\ref{thm:oo}, generates traces, whose first element is known as the follower and is targeted for $A$. $\rho$ interacts with joins such as $J_A: A\leq BC$ and $J_B: B\leq AC$, to \emph{extend} the length of its trace. As illustrated at the top of Figure~\ref{fig:traces-layer}, the initial trace of $\rho$ is written in the top row as ``$abababab$'', to depict balls whose targets alternate between $A$ and $B$, and where the first ``$a$'' is the follower of $\rho$ is targeted for $A$. To obtain this $ABAB$-trace, we start with the follower. Since it is targeted for $A$, we apply $J_A$ to determine that the target for the next ball must be $B$ or $C$. In our example we picked $B$. Then since $J_B$ exists, the ball extending the second one needs to be targeted for $A$ or $C$ according to $J_B$. In our example we chose $A$. This process repeats until the follower gets realized.

Generally, we pick targets that make the trace shorter, so as to demand less fickleness. One might then wonder why we did not choose to target $C$ in the initial trace, since no $J_C$ requirements exist, meaning that $C$ elements do not need to be extended. The reason is $C$ elements cannot be picked until the follower is realized, otherwise the enumeration of those elements later will un-realize the follower. Therefore we are forced to pick a growing $ABAB$-trace at first. The initial situation is similar for Theorem~\ref{thm:oo-original}, where the first trace is an $ABAB$-one, as illustrated at the bottom of Figure~\ref{fig:traces-layer}.

After realization, the $ABAB$-trace descends and gets stopped at the lowest priority gate $\preceq\rho$ that does not allow the entire trace to pass. Each gate represents a meet condition, for which there are two types in our theorem:
\begin{align*}
    (\forall W \leq A, B)\; [W \leq C],\\
    (\forall W \leq A, C)\; [W \leq B].
\end{align*}

We shall call these $AB$- and $AC$-gates respectively. Gates do not allow elements from both of its sides to be enumerated at the same time. In our theorem as shown at the top of the figure, lower priority gates are drawn further up, so the lowest priority gate is an $AB$-gate, preventing elements from $A$ and $B$ from passing simultaneously. Since the initial trace contains elements for $A$ and $B$, the trace will be \emph{stopped} by this gate. A similar argument can be made for the $ABAB$-trace at the bottom of the figure, which is also stopped by the lowest priority $AB$-gate.

When the gate \emph{opens}, only elements from the same side of the gate can pass. Therefore, we \emph{partition} the stopped trace into a head and a tail, making the tail as long as possible, but ensuring that when the elements in the tail are simultaneously enumerated, no more than one side of the gate is injured. In the figure, the last element of the trace is targeted for $B$, so that $b$ can pass when the $AB$-gate opens. The passing injures the $B$-side of the gate, thereby \emph{closing} the gate, which will only reopen when the injured side has recovered. The $A$-element preceding $b$ cannot pass with $b$ because that will injure also the $A$-side of the gate. Hence, the original trace $abababab$ is \emph{partitioned} into the $abababa$ head and the $b$ tail, where only the tail passes, leaving the head at the gate. The tail descends to the next gates. These gates will eventually allow $b$ to pass, since a single $B$-enumeration will not injure both sides of any gate. $b$ passes all the gates, and gets enumerated, as depicted by the cancellation of $b$ in the figure.

After the tail is enumerated, we return to its head $abababa$, which remains at the lowest priority $AB$-gate. This trace terminates with an $A$-element, so we shall refer to such traces as $ABABA$-traces. While waiting for the stopping gate to reopen, the joins force us to extend the trace. We apply again the trace \emph{extension} rules described earlier, adding to the rules the criteria that the extension's targets should avoid the elements at the side of the stopping gate that is just injured. For instance, since the last element of the $ABABA$-trace is targeted for $A$, by the extension rules we can extend by a $B$- or $C$-element, and by the new criteria we cannot choose $B$ since the $AB$-gate has just been injured at the $B$-side. So we are forced to extend by $C$. By the extension rules again, in the top example, since there are no $J_C$ conditions, the trace stops growing. But in the bottom example, the rules say we must continue extending by $A$ or $B$, and by the new criteria we can only choose $A$, and repeating the extension rules we get an $ACAC$-extension.

When the $AB$-gate reopens, the stopped trace stops extending. We apply the \emph{partitioning} rules again to break the stopped trace into a head and tail, making the tail as long as possible without containing elements targeted for both sides of the gate. The tail passes the $AB$-gate until it gets stopped again. In the top example, when the $AB$-gate reopens, the waiting trace is an $ABABAC$-one, which we can partition into an $ABAB$-head and an $AC$-tail according to the partitioning rules, since $A$ and $C$-elements do not appear on both sides of the $AB$-gate. The tail passes the gate but is stopped by the next $AC$-gate, which prevents $A$ and $C$-elements from passing simultaneously. A similar argument can be made for the bottom example, where trace is partitioned to a $ACAC$-tail.

We repeat the process with the trace that is just stopped --- waiting for the stopping gate to open; partitioning the waiting trace into a head and a tail; allowing the tail to pass till it is stopped by a higher priority gate; repeating the process with the tail as the new waiting trace if the tail is not yet enumerated; returning to tend to the head when the tail is gone; extending the head by the extension rules when that happens; letting the extended head be the new waiting trace; and continuing until the follower is enumerated.

Every enumeration requests for one permission; therefore the total length of the traces bounds the fickleness required. Consider the length of the traces for the conditions in Theorem~\ref{thm:oo}. In our example at the top of Figure~\ref{fig:traces-layer}, $\rho$ has three higher priority gates that alternate between $AB$ and $AC$. Following the trace extension and partitioning rules, the length of the initial trace that gets stopped at the lowest priority $AB$-gate is bounded by $\omega$, since the trace is an $ABAB$-one that stops extending only when the follower of $\rho$ gets realized, which could take arbitrarily long. At the middle $AC$-gate, the stopped trace could be an $AC$-trace or a $BC$-trace. Either way, the length does not exceed two. At the bottom $AB$-gate, the stopped trace is an $ABAB$-one, which keeps extending until the previous $AC$-gate reopens, and that again could take arbitrary long. Hence the trace's length is bounded by $\omega$.

Now each element of the trace at the bottom gate requests for one permission, so the number of requests by a trace of length $<\omega$ at the bottom is $<\omega$. But every $<\omega$-length trace at the bottom gate is extended from one element at the middle $AC$-gate. Since the trace at the middle gate is never more than two elements long, if we consider only the effect from the bottom two gates, the fickleness requested is bounded by $\omega\cdot 2$. Repeating this argument, every 2-element trace at the middle is generated from one element at the top gate, and since there are $>\omega$-elements at the top, the total length of all the traces of $\rho$ is bounded by $\omega\cdot 2\cdot \omega=\omega^2$, in our example where $\rho$ has only three higher priority gates. More generally, given that $\rho$ will not have more than $|\rho|$ many higher priority gates, the total length of the traces from $\rho$ is bounded by
\[\underbrace{\omega\cdot 2 \cdot\omega \cdot 2 \cdot\ldots \cdot \omega \cdot 2\cdot\omega}_{\leq|\rho| \text{ many products}} <\omega^{\lceil |\rho|/2 \rceil},\]
obtained from the alternating $ABAB$ then $AC$-traces of length $<\omega$ and $\leq 2$ respectively.

Similarly, with the addition of $J_C$ in Theorem~\ref{thm:oo-original}, the traces of $\rho$ alternate between $ABAB$- and $ACAC$-ones of length $<\omega$, as shown at the bottom of Figure~\ref{fig:traces-layer}, bounding the total length of traces by
\[\underbrace{\omega\cdot \omega \cdot\ldots \cdot\omega}_{\leq|\rho| \text{ many products}} <\omega^{|\rho|}.\]

Given that $|\rho|$ can get arbitrarily large, the bound of the fickleness for our theorem is still $\omega^\omega$, even though our traces do not grow at every other gate. The alternating $ABAB$- and $AC$-trace of our theorem guides the modification of the proof of Theorem~\ref{thm:oo-original}. In the original proof, roughly speaking, the authors set up $<\omega^\omega$-many layers to protect $\rho:A\not\leq C$. The layers alternate between $ABAB$- and $ACAC$-layers, each of length $<\omega$, to reflect the $ABAB$- and $ACAC$-traces generated by $\rho$. We tighten their proof by replacing every pair of $ABAB$-$ACAC$-layers by $ABAB$-$AC$-$ABAB$-layers, where the $ABAB$-layers are of length $<\omega$, while the middle $AC$-layer is an $AC$-pair. These $ABAB$-$AC$-$ABAB$ layers reflect the $ABAB$-$AC$-$ABAB$-traces.

We reuse notation from \cite{downey2020hierarchy}. We describe the construction for the case with just two alternations of $AB$- then $AC$-gates, where we show that no \re set $X$ with degree fickleness $\leq\omega^2$ can embed structures satisfying the conditions of Theorem~\ref{thm:oo}. The generalization from $\omega^2$ to $\omega^n$ for arbitrary $n\in\omega$ will not be different from described in \citep{downey2020hierarchy}, so we focus on $\omega^2$, which illustrates most of the changes needed. Fix functionals $\Lambda$, $\Phi_0$, $\Phi_1$ such that $\Lambda(X)=(A,B,C)$, $A=\Phi_0(B,C)$, and $B=\Phi_1(A,C)$. For $x\in\omega$, define
\begin{align*}
    x^{(1)} \defeq &\;\max\{\phi_0(x), \phi_1(x)\},\\
    x^{(n+1)} \defeq &\; \left(x^{(n)}\right)^{(1)},
\end{align*}
with the idea that a change in $A$ or $B$ below $x$ necessitates a change in $A$ or $B$ or $C$ below $x^{(1)}$. We think of $x^{(n)}$ as having $n$ \emph{layers} of protection above $x$, which our opponent $X$ can peel with all the fickleness that it has remaining. Our bottom-most layer can be thought of as an $A$-layer that is associated with an $A\leq C$-computation that we want to stop $X$ from peeling at all costs. If we set up enough layers, we can prevent $X$ from reaching that $A$-layer, which would have resulted in an $A$-change that diagonalized $A$ against a $C$-computation. Since $X$ has $\leq\omega^2$-fickleness, we set up $\omega^2$ many layers above the bottom one. But this strategy will only work if $X$ peels layers one at a time. If layers are lost quickly, then $X$ might reach the bottom layer with $\leq\omega^2$-fickleness. When double peeling occurs, we arrange to win by a meet requirement. For instance, if $A$ and $B$-layers are simultaneously peeled, we will receive the $AB$-changes needed to diagonalize some set $Q\leq A,B$ against another $C$-computable function, allowing us to win by $(\exists Q\leq A,B) [Q\nleq C]$. Similarly, if the $AC$-layers are peeled, we can win by $(\exists Q\leq A,C) [Q\nleq B]$.

Let $\Delta(X)$ be the $\leq\omega^2$-fickle function in our construction. To know the number of layers to set up at a given stage, we want to know the fickleness that remains in $\Delta(X)$. In other words, we want to know the $\leq\omega^2$-fickle function that computably approximates $\Delta(X)$. By the following lemma, there are only countably many approximations to guess from, and we can effectively enumerate them:
\begin{Lemma}[\cite{downey2020hierarchy}] \label{lemma:enum}
    Let $\alpha\leq \epsilon_0$. There exists a canonical $\mathcal{R}$ of order type $\alpha$, a computable function $g(e):\omega \to R$ known as the \emph{bounding function}, and uniformly computable functions $f^e(x,s):\omega^3\to\omega$ and $m^e(x,s):\omega^3\to R$ known as the \emph{computable approximation functions}, such that $\langle f^e(x,s), m^e(x,s)\rangle_s$ is a $g(e)$-computable approximation of $f^e(x) :=\lim_s f^e(x,s)$, and $\langle f^e\rangle_{e\in\omega}$ enumerates the $<\alpha$-fickle functions.
\end{Lemma}

We will not know which computable approximation works for $\Delta(X)$, so we simultaneously try all of them, letting the $e$-th \emph{agent} guess that $\Delta(X)$ is the $e$-th approximation $f^e,o^e$ as in the lemma, with $\Delta(X)=f^e$. At stage $s$, we keep track of the first $s$ steps of the first $s$ agents. As we shall see, the agents work independently because they construct their own sets and functionals that do not interfere with each other. If we guess that $\Delta(X)$ has $(\omega m_1+m_0)$ many fickleness remaining, we would like to have $m_1$, followed by $m_0$ many levels of protection. Each level involves a pair of elements, so we shall set up $(2m_1+1)$ inner $ABAB$-layers, followed by a middle $AC$-layer, followed by $(2m_0+1)$ outer $ABAB$-layers, as illustrated at the top of Figure~\ref{fig:layers}. The bottom-most $A$-layer works for $C\leq A$, the $ABAB$-layers work to construct a set $\leq A,B$ that cannot be computed by $C$, and the middle $AC$-layer works to construct a set $\leq A,C$ that cannot be computed by $B$.

\begin{figure}[tpb]
    \centering
    \tikzset{every picture/.style={thick}}
    \begin{tikzpicture}[every node/.style={rectangle,draw=none,fill=none}]
        \tikzset{pics/gate/.style n args={3}{code ={
            \node at (-1,0) {$#1$};
            \node[text width=10em,align=left] at (59+#2em,0.2) {$#3$};
            \draw (0,0)--(5,0);
            }}, pics/gate/.default={\eta}{0}{abab}
        }
        \tikzset{pics/down/.style n args={3}{code ={
            \draw[dashed] (#1,#2)--(#1,-2);
            \node at (#1,-2.2) {$#3$};
            }}, pics/down/.default={-2}{-1}{x}
        }
        \tikzset{pics/eq/.style n args={4}{code ={
            \node[rectangle,draw=none,fill=none] at (0,0) {$\begin{array}{c}
                \scriptstyle 2m_{#4}+1\\
                \scriptstyle A,#1\geq #2\neq #3
            \end{array}$};
            }}, pics/eq/.default={B}{Q_{ij}}{\Psi_k(C)}{1}
        }
        \tikzset{layer1/.pic ={
            \node at (0,0) {$ABAB\ldots ABAB$};
            \node at (1.1,0.5) {$AC$};
            \node at (2.2,1.0) {$ABAB\ldots ABAB$};
            \node at (0.0,-0.4) {$\underbrace{\hspace{2.6cm}}$};
            \path (-0.1,-1.0) pic{eq};
            \node at (2.2,1.3) {$\overbrace{\hspace{2.6cm}}$};
            \path (2.2,2.0) pic{eq={B}{Q}{\Psi_i(C)}{0}};
            \node at (1.1,2.6) {$\overbrace{\hspace{0.4cm}}^{A,C\geq Q_i \neq\Psi_j(B)}$};
            \path (0,0) pic{down={-1.3}{-0.5}{x_{ij}}};
            \path (0,0) pic{down={1.1}{-0.5}{u_{ij}<x_i<u_i<x}};
            \path (0,0) pic{down={3.4}{0.8}{u}};
            \node at (-1.4,0.5) {$\overbrace{\hspace{0.1cm}}^{\Xi_{ijk}(C)=A}$};
        }}
        \tikzset{layer2/.pic ={
            \node at (0,0) {$ABAB\ldots ABAB$};
            \node at (2.2,0.5) {$ACAC\ldots ACAC$};
            \node at (0.0,-0.4) {$\underbrace{\hspace{2.6cm}}$};
            \path (-0.1,-1.0) pic{eq={B}{Q}{\Psi_j(C)}{1}};
            \node at (2.2,0.8) {$\overbrace{\hspace{2.6cm}}$};
            \path (2.2,1.5) pic{eq={C}{Q}{\Psi_i(B)}{0}};
            \node at (-1.4,0.5) {$\overbrace{\hspace{0.1cm}}^{\Xi_{ij}(C)=A}$};
        }}
        \node at (-5,0) {Without $J_C: C\leq AB$};
        \path (0,0) pic{layer1};
        \node at (-5,-5) {With $J_C: C\leq AB$};
        \path (0,-5) pic{layer2};
    \end{tikzpicture}
    \caption{The top set of layers is obtained from the conditions in Theorem~\ref{thm:oo}, while the bottom set is from Theorem~\ref{thm:oo-original}, which has the additional join $J_C$. At the top, given that the opponent has $(\omega m_1+m_0)$ many fickleness remaining, we set up $(2m_1+1)$ inner $ABAB$-layers, followed by a middle $AC$-layer, followed by $(2m_0+1)$ outer $ABAB$-layers. Without $J_C$, the $AC$-layer cannot be longer. The $ABAB$-layers protect the $AB$-gate computations while the $AC$-layer protects the $AC$-gate computation. At the bottom, which has $J_C$, the $AC$-pair and second $ABAB$-layer can be replaced by a long $ACAC$-layer.}
    \label{fig:layers}
\end{figure}
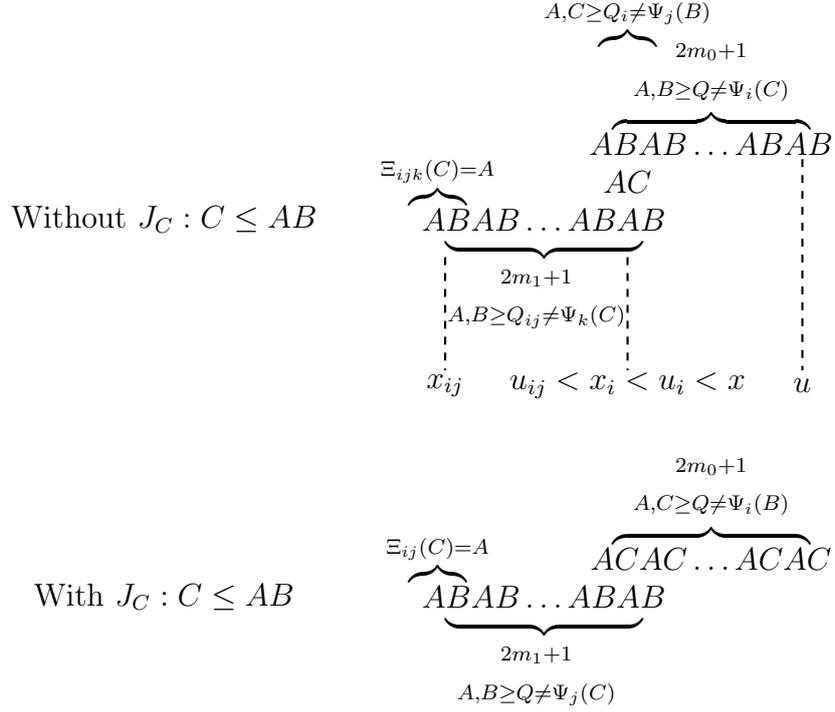

The $e$-th agent sets up its layers by enumerating an \re set $Q=Q^e$ and functionals $\Gamma=\Gamma^e$ and $\Theta=\Theta^e$ with the aim of having
\[\Gamma(A)=\Theta(B)=Q \neq \Psi_i(C)\]
for all $i\in\omega$. If we fail at some $i$, we enumerate a backup \re set $Q_i=Q^e_i$ and functionals $\Gamma_i=\Gamma^e_i$ and $\Theta_i=\Theta^e_i$, with the aim of having
\[\Gamma_i(A)=\Theta_i(C)=Q_i \neq \Psi_j(B)\]
for all $j\in\omega$. If we fail at some $j$, we enumerate a backup \re set $Q_{ij}=Q^e_{ij}$ and functionals $\Gamma_{ij}=\Gamma^e_{ij}$ and $\Theta_{ij}=\Theta^e_{ij}$, with the aim of having
\[\Gamma_{ij}(A)=\Theta_{ij}(B)=Q_{ij} \neq \Psi_k(C)\]
for all $k\in\omega$. If we fail at some $k$, we enumerate a functional $\Xi_{ijk}=\Xi^e_{ijk}$ with the aim of having
\[\Xi_{ijk}(C) = A.\]

The agent appoints followers $x$, $x_i$, $x_{ij}$ targeted for $Q, Q_i, Q_{ij}$ respectively. Each $x_{ij}$ may change several $x_i$'s and each $x_i$ may change several $x$'s. If $n$ is the code associated with these three followers, then at a stage where $o^e(n)=\omega m_1+m_0$, we shall arrange that
\begin{align*}
    u_{ij}= &\;\max\left\{x_{ij}^{(2m_1+1)}, \gamma_{ij}(x_{ij})\right\},\\
    u_i= &\;\max\left\{x_i^{(2)}, \gamma_i(x_i)\right\},\\
    u= &\;\max\left\{x^{(2m_0+1)}, \gamma(x)\right\},
\end{align*}
\begin{align} \label{eq:u}
    x_{ij} <u_{ij} <x_i <u_i <x <u <\lambda(u_x) <\delta(n),
\end{align}
as illustrated at the top of Figure~\ref{fig:layers}, where $x_{ij}$ denotes the position of the bottom-most layer, $u_{ij}$ denotes the end of the $(2m_1+1)$ inner $ABAB$-layers, $u_i$ denotes the end of the middle $AC$-layers, while $u$ denotes the end of the $(2m_0+1)$ outer $ABAB$-layers. The $e$-th agent enumerates its version of $Q, Q_i, Q_{ij}$, so there is no conflict with other agents. We often drop the subscript $e$ for conciseness.

When the opponent attacks, it loses one fickleness and peels one or more layers. If the middle layers are attacked, resulting in $AB$, $BA$, or $AC$-changes, then we can win by the gate associated with those changes. Consider what happens when the opponent peels one layer at a time, or in order words, peels only the outermost layer whenever it attacks. What happens after the outer set of $ABAB$-layers, then the middle pair of $AC$-layers, which are those above $x_i$, have been peeled? The opponent's remaining fickleness would have descended to the next limit ordinal $\omega m_1$. Then when the opponent next attacks, it will peel the $B$-layer below $x_i$, and its ammunition would fall below the limit ordinal, say to $\omega (m_1-1) +m_0'$. The $B$-change causes $\Psi_j(B,x_i)$ to diverge, or, in other words, causes $x_i$ to be un-realized, which allows us to set up a new pair of middle $AC$-layers and $(2m_0'+1)$ outer $ABAB$-layers above the pair. This makes us ready for $(\omega (m_1-1) +m_0')$ many attacks. In other words, no more than two layers below $x_i$ will be lost before we can set up layers again, which makes the $(2m_0+1)$ many layers above $x_{ij}$ sufficient against our opponent.

In fuller detail, we pick large followers $x_{ij}<x_i<x$ satisfying Equation~\eqref{eq:u} and wait for the followers to be realized. While waiting, we re-pick $x_i$ and $x$, if necessary, to preserve the equation. In the meantime, we also update 
\[\delta(n)=\lambda(u).\]
After the followers are realized, set
\begin{align} \label{eq:psi}
    \xi(x_{ij}) &=\max\{u_{ij},\psi_k(x_{ij})\},\\
    \nonumber \theta_{ij}(x_{ij}) &=\max\{u_i,\psi_j(x_i)\},\\
    \nonumber \theta_i(x_i) &=\max\{u,\psi_i(x)\},\\
    \nonumber \theta(x) &=u.
\end{align}


Note that $\xi$ ($\theta_{ij}$, $\theta_i$) can be maintained because it only needs to exist if $\Psi_k$ (respectively $\Psi_j$, $\Psi_i$) is total. Consider the possible ways that the opponent may peel $x_{ij}'s$ $(2m_1+1)$ many layers that lie between $x_{ij}$ and $x_i$ and that are $\leq u_{ij}$. If $C$ changes below $u_{ij}$, then $\xi(x_{ij})\uparrow$ and $\delta(n)\uparrow$, which allows us to cancel all three followers and pick new large ones. If $C$ changes below $\psi_k(x_{ij})$ but not below $u_{ij}$, then $x_{ij}$ is un-realized and $\delta(n)\uparrow$, so we return to the earlier phase of waiting for realization to update Equations~\eqref{eq:psi}. If an inner $A$-layer is peeled ($\gamma_{ij}(x_{ij})\uparrow$), which results in the next $B$-layer being peeled ($\theta_{ij}(x_{ij})\uparrow$), then the simultaneous $AB$-change allows us to \emph{attack with $x_{ij}$} by enumerating $x_{ij}$ into $Q_{ij}$ and winning by $\Gamma_{ij}(A,x_{ij}) =\Theta_{ij}(B,x_{ij}) =Q_{ij}(x) \neq \Psi_k(C,x_{ij})$. If an $A$-change resulted instead in a $C$-change, then two of $x_i$'s layers will be peeled, since $\gamma_i(x_i),\theta_i(x_i)\geq x_i>u_{x_{ij}}$. This allows us to attack with $x_i$. If $x_{ij}$'s outer-most $B$-layer is peeled, then the $\theta_{ij}(x_{ij})$ and $\delta(n)$ changes allow us to cancel $x_i$ and $x$ and set up new large ones. Finally, if $x_{ij}$'s outer-most $A$-layer is peeled, then $x_{ij}$ indeed loses a layer.

Next, consider $x_i's$ two layers, which lie between $x_i$ and $x$ and that are $\leq u_i$. Similar to the smaller follower, $B$-changes below $u_i$ allow us to cancel $x_i$ and $x$, while $B$-changes below $\psi_j(x_i)$ allow us to wait for realization of $x_i$ again. Also, if the inner $A$-layer is peeled ($\gamma_i(x_i)\uparrow$), which results in the next $C$-layer to be peeled ($\theta_i(x_i)\uparrow$), then we can \emph{attack with $x_i$}. If an $A$-change resulted instead in a $B$-change, then two of $x$'s layers will be peeled, since $\gamma(x),\theta(x)\geq x>u_{x_i}$. This allows us to attack with $x$. If $x_i$'s outer-most $C$-layer is peeled, we can cancel $x$ and set up a new large one. Finally, if $x_i$'s outer-most $A$-layer is peeled, then $x_i$ indeed loses a layer.

Finally, consider $x's$ $(2m_0+1)$ many layers that lie between $x$ and $\leq u$. Similar to the smallest follower, $B$-changes below $u$ allow us to cancel $x$, while $C$-changes below $\psi_i(x)$ allow us to wait for realization of $x$ again. Also, if the inner $A$-layer is peeled ($\gamma(x)\uparrow$), which results in the next $B$-layer to be peeled ($\theta(x)\uparrow$), we can \emph{attack with $x$}. If an $A$-change resulted instead in a $C$-change, then this change below $u$ allows us to cancel $x$, as just mentioned. If $x$'s outer-most $A$ or $B$-layer is peeled, then $x$ indeed loses a layer.

Note that our two-gate construction is similar to the three-gate one of \citep{downey2020hierarchy}. But we need to be careful when inner $C$-layers are peeled, since we will not be able to obtain a $CA$- or $CB$-change without $J_C: C\leq BA$, and these changes could have helped us to win at a gate. The paragraphs above hint at why $J_C$ was not needed: When we obtain a solo $C$-change in the middle, we can cancel followers and set up new layers, to the effect of not losing layers.

Consider the trickier situation where layers between $x$ and $u$ are still unpeeled, but the $C$-layer below $u_i$ which belongs to $x_i$ is peeled ($\theta(x_i)\uparrow$). One might worry that the $A$-layer below the peeled layer will now become vulnerable. We argue that this is not the case: The $C$-change allows us to cancel $x$ and set up $(2m_0+1)$ many new layers above a new $x$. We are also allowed to lift $\theta_i(x_i)$ above the new $u$ and $\psi_i(x)$. Assume there are attacks on the $A$-layer ($\gamma(x_i)\uparrow$) below the lost $C$-layer due to a $B$-change. If this $A$-change occurred before the new $x$ is realized, then we can attack with $x_i$, since $\gamma(x_i)$ will not have converged yet. If the $A$-change occurred after realization, we will be able to attack with the new $x$ because the $AB$-change would be below the new $u_i$, which is in turn $\leq x\leq \gamma(x), \theta(x)$. Another concern is if the $B$-change occurred without the $A$-layer below being peeled away. The $B$-change would exceed $\theta_{ij}(x_{ij})$, making us unable to cancel $x_i$, as we would have before the $C$-layer was gone. The $B$-change must lead to an $A$- or $C$-change on the layer above. In the former case, we can attack with $x$, and in the latter case, we can cancel $x$ and set up new layers, as in the situation when the original $C$-layer was peeled.

A symmetrical argument can be made to show how the outer-most $A$-layer of $x_{ij}$ cannot be peeled as long as there are layers belonging to $x_i$ sitting above. We are also prepared against attacks on the middle layers. For instance, if the opponent attacked the layers of $x_{ij}$ ($x$), we can obtain changes in $C$, $AC$, $BC$, $AB$, or in $BA$. In the first three cases, we may cancel $x_{ij}$ (resp. $x$), and in the last two cases we may attack with $x_{ij}$ (resp. $x$). If we attacked the middle layers of $x_i$, changes in $AC$ allow us to attack with $x_i$, changes in $C$ or $BC$ allow us to cancel $x$, while changes in $AB$ or $BA$ allow us to win by $x$.

The precise verification of this 2-gate construction is similar to the 3-gate construction of \citep{downey2020hierarchy}. To generalize to the $n$-gate case, the generalization in \citep{downey2020hierarchy} still applies: If we guess that $\Delta(X)$ is $\omega^r$-fickle, so that
\[o^e(n)=\omega^{r-1}m_{r-1} +\ldots +\omega m_1 +m_0,\]
we set up $2r-1$ many followers that alternate between having $AB$- and $AC$-layers, and a follower associated with $AC$ has only two $AC$-layers above it, while the $k$-th largest follower associated with $AB$ has $2m_k+1$ $AB$-layers above it.

\section{Rejecting some ``Larger'' Lattices Quickly} \label{sec:reject}
Since the $\leq3$-direct lattices are not $>\omega^2$-lattices, we considered larger structures as candidates. Given that $L_7$ characterizes the $>\omega$-levels, it was speculated that we could reach the $>\omega^2$-levels by making two $L_7$'s interact with each other. The two candidates in Figure~\ref{fig:lempp-lerman} were suggested as a result. However, we can use earlier results to reject these candidates immediately: Theorem~\ref{thm:oo} says that any structure that satisfies the conditions of the theorem demands at least $\omega^\omega$-fickleness and therefore cannot be a candidate $>\omega^2$-lattice. In particular, since the four lattices in the top row of Figure~\ref{fig:3direct} satisfy those conditions, we get:
\begin{Corollary} \label{cor:non-candidate}
    A structure cannot characterize $>\omega^2$-fickleness if it contains, as sublattice, any of the four lattices in the top row of Figure~\ref{fig:3direct}.
\end{Corollary}

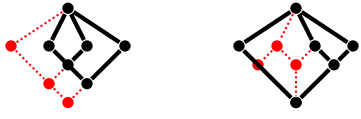
\begin{figure}[tpb]
    \centering
    \tikzset{every picture/.style={thick}}
    \begin{tikzpicture}[every node/.style={circle,fill=black,inner sep=1.5pt}]
        \tikzset{lempp/.pic ={
            \node[fill=red] (0) at (0,-1) {};
            \node (1) at (0,1.5) {};
            \node[fill=red] (A) at (-0.5,-0.5) {};
            \node (X) at (0.5,-0.5) {};
            \node (AX) at (0,0) {};
            \node (BX) at (-0.5,0.5) {};
            \node (AY) at (0.5,0.5) {};
            \node[fill=red] (C) at (-1.5,0.5) {};
            \node (Z) at (1.5,0.5) {};
            \draw[ultra thick] (1) --(Z) --(X) --(BX) --(1) --(AY) --(AX);
            \draw[densely dotted,red] (1) --(C) --(A) --(0) --(X);
            \draw[densely dotted,red] (AX) --(A);
        }}
        \tikzset{lerman/.pic ={
            \node (0) at (0,-1) {};
            \node (1) at (0,1.5) {};
            \node (AY) at (0.5,0.5) {};
            \node (C) at (-1.5,0.5) {};
            \node (Z) at (1.5,0.5) {};
            \node (Y) at (1,0) {};
            \node (X) at (0,-1) {};
            \node[fill=red] (AX) at (0,0) {};
            \node[fill=red] (BX) at (-0.5,0.5) {};
            \node[fill=red] (B) at (-1,0) {};
            \draw[ultra thick] (1) --(C) --(0) --(Y) --(Z) --(1) --(AY) --(Y);
            \draw[densely dotted,red] (1) --(BX) --(AX) --(X);
            \draw[densely dotted,red] (AY) --(AX);
            \draw[densely dotted,red] (BX) --(B);
        }}
        \tikzset{large/.pic ={
            \node[red] (abc) at (0,4) {};
            \node[red] (ab) at (-4.5,3) {};
            \node[red] (ac) at (0,3) {};
            \node[red] (bc) at (4.5,3) {};
            \node[red] (abnac) at (-4.5,2) {};
            \node[red] (abnbc) at (0,2) {};
            \node[red] (acnbc) at (4.5,2) {};
            \node (abnacnbc) at (0,1) {};
            \node[red] (a) at (-4.5,0) {};
            \node[red] (b) at (0,0) {};
            \node[red] (c) at (4.5,0) {};
            \node (anbc) at (-4.5,-1) {};
            \node (bnac) at (0,-1) {};
            \node (cnab) at (4.5,-1) {};
            \node (0) at (0,-2) {};
            \draw[densely dotted,red] (abc)--(ab)--(abnac)--(a)--(anbc)--(0)--(cnab)--(c)--(acnbc)--(bc)--(abc)--(ac)--(abnac)--(abnacnbc)--(acnbc)--(ac);
            \draw[densely dotted,red] (ab)--(abnbc)--(bc) (abnbc)--(abnacnbc)--(anbc) (abnacnbc)--(cnab) (b)--(bnac)--(0);
            \draw[densely dotted,red] (b) edge[bend left] (abnbc);
            \draw[ultra thick] (bnac) edge[bend right] (abnacnbc);
            \draw[ultra thick] (bnac)--(0)--(anbc)--(abnacnbc)--(cnab)--(0);
        }}
        \path (-3,0) pic[scale=0.5]{lempp};
        \path (0,0) pic[scale=0.5]{lerman};
    \end{tikzpicture}
    \caption{These lattices cannot characterize $>\omega^2$-fickleness by Corollary~\ref{cor:non-candidate}. The overly fickle sublattice is shown in solid black lines. The lattices were suggested by Steffen Lempp and Manuel Lerman respectively.}
    \label{fig:lempp-lerman}
\end{figure}

\section{Some ``Short'' and ``Narrow'' Upper Semilattices (USLs)} \label{sec:usl}
We also considered USLs $\mathcal{U} =(U,\leq,\cup,\cap)$ as $>\omega^2$-candidates. In a USL, the meet function $\cap$ may not be total. We obtained USLs by breaking some meets of the finite lattices considered earlier. For example, if we removed the meet of $M_3$, we would get the right most USL at the bottom row of Figure~\ref{fig:usls}, where the removed meet is depicted by an open circle. The figure also shows other USLs based on lattices in Figure~\ref{fig:3direct}. We can show that:

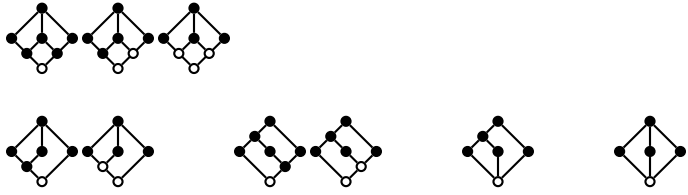
\begin{figure}[tpb]
    \centering
    \tikzset{every picture/.style={thick}}
    \begin{tikzpicture}[every node/.style={circle,draw=black,fill=black,inner sep=1.2pt}]
        \tikzset{pics/oo0/.style n args={2}{code ={
            \node[fill=#2] (d) at (0,-1) {};
            \node (l) at (-1,0) {};
            \node (u) at (0,1) {};
            \node (r) at (1,0) {};
            \node (c) at (0,0) {};
            \node[fill=#1] (dl) at (-0.5,-0.5) {};
            \draw (dl) --(c) --(u) --(r) --(d) --(dl) --(l) --(u);
            }}, pics/oo0/.default={black}{black}
        }
        \tikzset{pics/oo0s/.style n args={2}{code ={
            \path (#2,0) pic[scale=#1]{oo0={black}{white}};
            \path (#2*2,0) pic[scale=#1]{oo0={white}{white}};
            }}, pics/oo0s/.default={1.0}{3}
        }
        \tikzset{pics/oo1/.style n args={2}{code ={
            \node[fill=#2] (d) at (0,-1) {};
            \node (l) at (-1,0) {};
            \node (u) at (0,1) {};
            \node (ul) at (-0.5,0.5) {};
            \node (r) at (1,0) {};
            \node (c) at (0,0) {};
            \node[fill=#1] (dr) at (0.5,-0.5) {};
            \draw[thick] (dr) --(c) --(ul) --(u) --(r) --(dr) --(d) --(l) --(ul);
            }}, pics/oo1/.default={black}{black}
        }
        \tikzset{pics/oo1s/.style n args={2}{code ={
            \path (#2,0) pic[scale=#1]{oo1={black}{white}};
            \path (#2*2,0) pic[scale=#1]{oo1={white}{white}};
            }}, pics/oo1s/.default={1.0}{3}
        }
        \tikzset{pics/oo2/.style n args={2}{code ={
            \node[fill=#1] (d) at (0, -1) {};
            \node (lu) at (-0.5, 0.5) {};
            \node (l) at (-1, 0) {};
            \node (u) at (0, 1) {};
            \node (r) at (1, 0) {};
            \node (c) at (0, 0) {};
            \draw (lu)--(c)--(d)--(l)--(lu)--(u)--(r)--(d);
            }}, pics/oo2/.default={black}{0}
        }
        \tikzset{pics/oo2s/.style n args={2}{code ={
            \path (#2,0) pic[scale=#1]{oo2={white}{0}};
            }}, pics/oo2s/.default={1.0}{3}
        }
        \tikzset{pics/m3/.style n args={2}{code ={
            \node[fill=#1] (a) at (0, -1) {};
            \node (b) at (-1, 0) {};
            \node (c) at (0, 1) {};
            \node (d) at (1, 0) {};
            \node (e) at (0, 0) {};
            \draw (a) -- (b) -- (c) -- (d) -- (a) -- (e) -- (c);
            }}, pics/m3/.default={black}{0}
        }
        \tikzset{pics/m3s/.style n args={2}{code ={
            \path (#2,0) pic[scale=#1]{m3={white}{0}};
            }}, pics/m3s/.default={1.0}{3}
        }
        \tikzset{pics/l7/.style n args={3}{code ={
            \node[fill=#3] (a) at (0, -1) {};
            \node (b) at (-1, 0) {};
            \node (c) at (0, 1) {};
            \node (d) at (1, 0) {};
            \node (e) at (0, 0) {};
            \draw[thick] (a) -- (b) -- (c) -- (d) -- (a);
            \node[fill=#1] (f) at (-0.5, -0.5) {};
            \node[fill=#2] (g) at (0.5, -0.5) {};
            \draw[thick] (c)--(e)--(f) (e)--(g);
            }}, pics/l7/.default={black}{black}{black}
        }
        \tikzset{pics/l7s/.style n args={2}{code ={
            \path (#2,0) pic[scale=#1]{l7={black}{black}{white}};
            \path (#2*2,0) pic[scale=#1]{l7={black}{white}{white}};
            \path (#2*3,0) pic[scale=#1]{l7={white}{white}{white}};
            }}, pics/l7s/.default={1.0}{3}
        }
        \path (0,1.5) pic{l7s={0.4}{1.0}};
        \path (0,0) pic{oo0s={0.4}{1.0}};
        \path (3,0) pic{oo1s={0.4}{1.0}};
        \path (6,0) pic{oo2s={0.4}{1.0}};
        \path (8,0) pic{m3s={0.4}{1.0}};
    \end{tikzpicture}
    \caption{Here are some USLs obtained by removing some meets of lattices in Figure~\ref{fig:3direct}. Removed meets are shown as open circles. For instance, the USLs at the top are based on $L_7$ (Figure~\ref{fig:l7}). Each USL turns out to characterize the same degrees as its lattice. Specifically, those at the top characterize $>\omega$-fickleness, while those below characterize $\geq\omega^\omega$-fickleness.}
    \label{fig:usls}
\end{figure}

\begin{Theorem} \label{thm:usls}
    The USLs in Figure~\ref{fig:usls} characterize the same \re degrees as the lattice they are based on.
\end{Theorem}

In particular, the USLs considered do not characterize $>\omega^2$-fickleness. From the theorem, we conjecture:
\begin{Conjecture} \label{conj:usl-equals-lattice}
    Let $U$ be a USL obtained by removing some meets of a finite lattice $L$. An \re degree bounds $U$ if and only if it bounds $L$.
\end{Conjecture}

Note that our USLs are different from the \emph{partial lattices} in \citep{lempp2006embedding}, where the authors removed not just the meet but also requirements associated with meets, producing structures that are embeddable below all nonzero \re degrees. Our USLs retain some form of meet, making them more difficult to embed.

The USL in Theorem~\ref{thm:usls} that takes most work is the one enlarged in Figure~\ref{fig:usl-m3}, which is based on $M_3$ (Figure~\ref{fig:N5-131}). In the rest of this section, we prove the theorem for this USL. Details for the easier USLs can be found in \citep{ko2021thesis}.

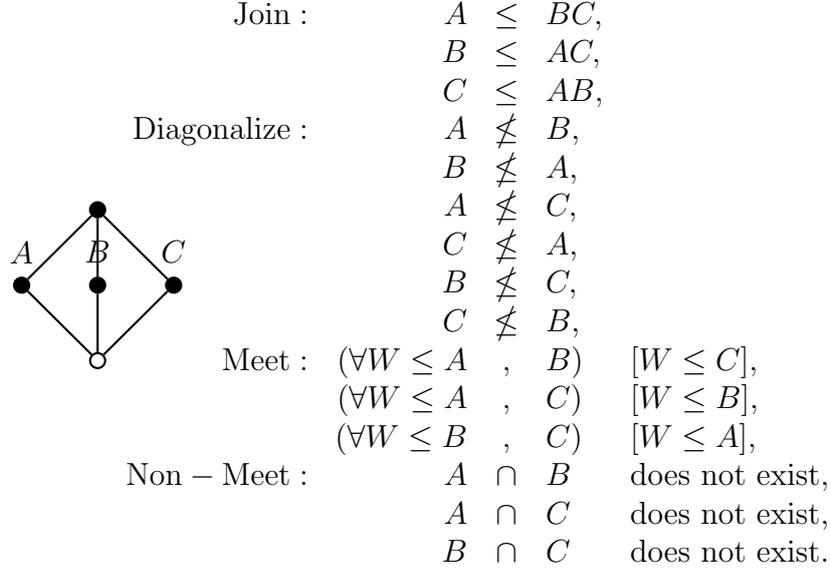
\begin{figure}[tpb]
    \centering
    \tikzset{every picture/.style={thick}}
    \begin{tikzpicture}[every node/.style={circle,fill=black,draw=black,inner sep=2.0pt}]
        \tikzset{131/.pic ={
            \node[fill=white] (d) at (0,-1) {};
            \node[label=above:{$A$}] (l) at (-1,0) {};
            \node (u) at (0,1) {};
            \node[label=above:{$C$}] (r) at (1,0) {};
            \node[label=above:{$B$}] (c) at (0,0) {};
            \draw[thick] (d) --(c) --(u) --(r) --(d) --(l) --(u);
        }}
        \tikzset{pics/eq/.style n args={2}{code ={
            \node[rectangle,draw=none,fill=none] at (0,0) {$\begin{array}{rrcll}
                \mathrm{Join: } &A &\leq &BC,\\
                &B &\leq &AC,\\
                &C &\leq &AB,\\
                \mathrm{Diagonalize: } &A &\nleq &B,\\
                &B &\nleq &A,\\
                &A &\nleq &C,\\
                &C &\nleq &A,\\
                &B &\nleq &C,\\
                &C &\nleq &B,\\
                \mathrm{Meet: } &(\forall W\leq A &, &B) &[W\leq C],\\
                &(\forall W\leq A &, &C) &[W\leq B],\\
                &(\forall W\leq B &, &C) &[W\leq A],\\
                \mathrm{Non-Meet: } &A &\cap &B &\text{does not exist},\\
                &A &\cap &C &\text{does not exist},\\
                &B &\cap &C &\text{does not exist}.
            \end{array}$};
            }}, pics/eq/.default={x}{y}
        }
        \path (0,0) pic{131};
        \path (5,0) pic{eq};
    \end{tikzpicture}
    \caption{This USL is obtained from removing the meet of $M_3$. The removed meet is shown as an open circle, and the requirements for bounding this USL are shown on the right. Like $M_3$, this USL characterizes the \re degrees that contain $\geq\omega^\omega$-fickle sets.}
    \label{fig:usl-m3}
\end{figure}

For the easier direction of the proof, we need to show that any \re degree that bounds the USL must contain a set of fickleness $\geq\omega^\omega$. Consider the requirements involved, shown at the right of Figure~\ref{fig:usl-m3}. The join ($J$) and diagonalization ($\rho$) requirements are exactly the same as the original lattice. As for the meets/gates $\eta$, there are three types:
\begin{align*}
    \eta_{W,C}: (\forall W \leq A,B) [W\leq C],\\
    \eta_{W,B}: (\forall W \leq A,C) [W\leq B],\\
    \eta_{W,A}: (\forall W \leq B,C) [W\leq A].
\end{align*}

We often drop the $W$ subscript. The computations associated with gate $\eta_A$ need to be computable by $A$, so we say that $\eta_A$'s \emph{computing set} is $A$. Similarly, the computing set of $\eta_B$ ($\eta_C$) is $B$ (resp. $C$). These gates are not strictly meets because $A\cap B$, $A\cap C$, and $B\cap C$ do not exist, but we still call them meets/gates because their strategies are the same as the $\eta$-strategies of Section~\ref{sec:sufficient}. Now $J$, $\rho$, and $\eta$ satisfy the conditions in Theorem~\ref{thm:oo}, so $\geq\omega^\omega$-fickleness is necessary to bound the USL, which completes the easier direction of the proof.

For the converse, fix an \re set $X$ whose degree contains a set $\Delta(X)$ of fickleness $\geq\omega^\omega$. We want to construct \re sets $A,B,C$ that satisfy the joins
\begin{align*}
    J_A: &A\leq BC,\\
    J_B: &B\leq AC,\\
    J_C: &C\leq AB,
\end{align*}
the three types of gates $\eta_A$, $\eta_B$, and $\eta_C$, and the six types of $\rho$ requirements. The strategies for $J$ and $\eta$ were described in Section~\ref{sec:sufficient}. The interactions between $J$, $\rho$, and $\eta$ are like in $M_3$, which was shown in \citep{downey2020hierarchy} to produce traces of length $<\omega^\omega$, making $\geq\omega^\omega$-fickleness sufficient for satisfying $\rho$.

Finally, consider the ``\emph{non-meet}'' requirements. There are three types, and they exert the non-existence of $A\cap B$, $A\cap C$ and $B\cap C$ respectively. We discuss only the first type, since the strategies of the others are the same after swapping the roles of $A$, $B$, and $C$. Given an \re set $V<A,B$ via $\Gamma=(\Gamma_A,\Gamma_B)$, we construct \re $E=E_{V\Gamma}$ via functional and $\Lambda=(\Lambda_A,\Lambda_B,\Lambda_C) =(\Lambda_{V\Gamma,A}, \Lambda_{V\Gamma,B},\Lambda_{V\Gamma,C})$, such that for all $\Pi$:
\begin{align*}
    \tau_{V\Gamma}:\; &V =\Gamma_A(A)=\Gamma_B(B) \implies \exists E, \Lambda_A, \Lambda_B, \Lambda_C \text{ such that}\\
    &E =\Lambda_A(A) =\Lambda_B(B) =\Lambda_C(C), \text{ and }\\
    \tau_{V\Gamma\Pi}:\; &\Pi(V) \neq E.
\end{align*}

When context is clear, we drop subscripts $V\Gamma$. We use $\tau$ and $\tau_\Pi$ to refer to the two requirements, and call them parent and child nodes respectively. The hypothesis of the parent says that $V\leq A,B$. If this holds, then $E$ witnesses that $V\neq A\cap B$, since $E$ is a set below $A$ and $B$ that cannot be computed by $V$ regardless of the functional $\Pi$ used for the $V$-computation. Note that we have assumed $V$ to be \renospace; this is enough because the meet of \re degrees, if it exists, must also be \renospace, according to this lemma:
\begin{Lemma}[\cite{lachlan1966lower}]
    If $A$ and $B$ are \re sets, and $U\leq_{\turing} A,B$ (not necessarily \renospace), then there exists an \re set $V\leq_{\turing} A,B$ such that $U\leq V$.
\end{Lemma}

Note that $\tau$ and its countably many children $\tau_\Pi$ also exert the non-existence of $A\cap B\cap C$, since any set below $A$ and $B$ is also below $A,B,C$, and the \re set $E$ constructed lies not just below $A$ and $B$, but also below $C$.

Fix a computable ordering of the requirements so that child nodes are always of lower priority than their parent. We can ignore a child if the hypothesis $V=\Gamma_A(A)=\Gamma_B(B)$ of its parent does not hold. We use the usual notion of $\tau$-expansionary stages to guess the outcome of a parent's hypothesis. Specifically, define the \emph{length} of $\tau$ at stage $s$ as
\begin{align*}
    \length(\tau)[s] &:=\max_{v\leq s}\; \left\{V\restriction v =\Gamma_A\restriction v =\Gamma_B\restriction v[s]\right\},
\end{align*}
and we guess that $\tau$ has outcome $\infty$ at stage $s$ if $s$ is an \emph{$\tau$-expansionary} stage, which is when
\begin{align*}
    \length(\tau)[s] &>\max_{s'<s}\{\length(\tau)[s']: s' \mathrm{ is a } \tau\mathrm{-stage}\}.
\end{align*}

Otherwise, we guess that the outcome of $\tau$ is $\fin$. By $\tau$-expansionary stages, we must increase our definitions of $E$, $\Lambda_A$, $\Lambda_B$, and $\Lambda_C$, so that if $\tau$'s true outcome is $\infty$, then $E=\Lambda_A(A)=\Lambda_B(B)=\Lambda_C(C)$:
\begin{framed}
    \noindent\textbf{$\tau=\tau_{V\Gamma}$-strategy}: \emph{Initialize} $\tau$ by setting $\Lambda_A=\Lambda_B=\Lambda_C=\emptyset$, $x=0$.
    \begin{enumerate}
        \item \emph{Maintain}: Wait for a $\tau$-expansionary stage. For each $e<x$, if $\lambda_A(e)$ or $\lambda_B(e)$ or $\lambda_C(e)$ is not defined, pick a large value for these uses, which are targeted for $A$, $B$, and $C$ respectively. Return to the beginning of this maintenance phase for $e+1$.
    \end{enumerate}
\end{framed}

We refer to $\lambda_A(e)$, $\lambda_B(e)$, and $\lambda_C(e)$ as the $A$, $B$, and $C$-uses of $e$ respectively. Each of these uses may be lifted only finitely often. Now consider the $\tau_\Pi$-strategy, which is based on the usual Friedberg Muchnik one, and as we shall see, is a finite injury requirement that inflicts both positive and negative injury. If $\tau_\Pi \succeq\tau^\frown\fin$, we can ignore $\tau_\Pi$. Henceforth, assume $\tau_\Pi \succeq\tau^\frown\infty$. Pick a large follower $e$ targeted for $E$, and (re)assign $e$ large $A$, $B$, and $C$-uses by $\tau$-expansionary stages. Note that these stages are not necessarily $\tau_\Pi$-stages, since $\tau^\frown\infty\preceq\tau_\Pi$. Therefore, we need to be mindful of the interactions between $\tau_\Pi$ and $\eta$ later, if $\tau\prec\eta\prec\tau_\Pi$, for $\tau_\Pi$'s uses may not be picked at $\eta$-expansionary stages, making the uses small enough to injure some of $\eta$'s earlier subrequirements. Next, we wait for $e$ to be \emph{realized} ($\Pi(e)=0$). We would like to diagonalize at $e$ by enumerating $e$ with its uses simultaneously. But to avoid un-realizing $e$, we need:
\begin{align*}
    \lambda_A(e) >\gamma_A(\restriction \pi(e)+1),\\
    \mathrm{or }\;\;\; \lambda_B(e) >\gamma_B(\restriction \pi(e)+1).
\end{align*}

In other words, we want to first lift the $A$- or $B$-use above its respective $\gamma$-use. Without loss of generality, we lift the $A$-use. To prevent the $B$-side of the computation from being lost, which could cause $e$ to be un-realized, we impose a restraint of $\gamma_B(\pi(e)+1)$ on $B$ prior to lifting. After lifting, by the next $\tau$-expansionary stage, the computation on the $\Gamma_A$-side would have returned, and we would pick a new large value for the $A$-use $\lambda_A(e)$, to exceed the $\gamma_A$-use. We can now change our restraint to the $A$-side by cancelling our restraint on $B$ and preventing elements below the raised $\gamma_A(\pi(e)+1)$ from entering $A$. Finally, we wait for $X$ to grant permission to enumerate $e$ with its new uses simultaneously. The $B$-enumeration may injure the $\gamma_B$-side of the computation, but $e$ will not become un-realized, because we are holding the $\gamma_A$-side of the computation, and if the $\gamma_B$-side does not return to equal the $\gamma_A$-side, then $\tau$ wins finitely and vacuously and we can ignore $\tau_\Pi$. Notice that each $e$ sets restraints and enumerates uses no more than twice. Therefore, $\tau_\Pi$ is a finite injury requirement, allowing it to interact nicely with other finitary requirements such as $\rho$.

What about interactions with infinitary requirements such as gates? First consider the easier case where a gate $\eta$ lies between $\tau$ and $\tau_\Pi$; for instance, if:
\[\tau^\frown\infty \preceq \eta_A^\frown\infty \preceq \tau_\Pi.\]
Recall that the strategy of a subrequirement $\eta_A,m$ involves enumerating an element $\leq\theta(m)$ into $A$, if ``small'' elements are enumerated into $B$ and $C$ simultaneously, which injured the $B$ and $C$-sides of the subrequirement. Now simultaneous enumeration occurs only at the end of the $\tau_\Pi$-strategy, when $e$ enumerates its uses. We need to ensure that if the $B$ and $C$-uses are small, then $\theta(m)$ can be picked to exceed the $A$-use. There are two cases. In the easier case, $\eta_A,m$ is set up before $e$. Then at least one of $e$'s $B$ and $C$-uses will be too large to injure the computations of $\eta_A,m$. In the harder case, the subrequirement is set up later. But because we only tend to the subrequirement at $\eta_A$-expansionary stages, which are also $\tau$-expansionary stages where $e$'s uses exist, $\eta_A,m$ will always know the latest values of those uses, and in particular, update $\theta(m)$ to exceed the latest $A$-use. Then the enumeration of the final $A$-use will lift $\theta(m)$, as desired. Likewise, if $\eta$ works for $\eta_B$ ($\eta_C$), then because $\tau\prec\eta$, $\eta$ can keep track of the latest $B$-use ($C$-use) of $e$, and delay believing computations until that use and $e$ have been enumerated.

The situation is trickier when $\tau_\Pi$ interacts with infinitary gates $\eta^\frown\infty \preceq\tau$ of higher priority than $\tau$; for instance, if:
\[\eta_A^\frown\infty \preceq \tau^\frown\infty \preceq\tau_\Pi.\]
Recall in our sketch of the $\tau_\Pi$-strategy that $e$ lifted its $A$-use once. A non-trivial situation occurs when a subrequirement $\eta_A,m_A$ is set up after $e$'s $A$-use is enumerated, but before stage $s_0$, the next $\tau$-expansionary stage when the new use is picked. $e$'s $B$ and $C$-uses may be small enough to injure both sides of the subrequirement, so we want $\theta(m_A)$ to exceed $e$'s $A$-use, should $e$'s current $B$ and $C$-uses be the final ones to be enumerated together later. But $\theta(m_A)$ may not be large enough, because it must be picked by the next $\eta_A$-expansionary stage, which generally comes before stage $s_0$. Therefore one of $e$'s $B$- or $C$-uses must be lifted after stage $s_0$, to avoid injuring $\eta_A,m_A$ at both sides.

Without loss of generality, suppose we lift the $B$-use at some stage $s_1>s_0$. What happens if there is an $\eta_B$-gate of higher priority than $\eta_A$:
\[\eta_B^\frown\infty \preceq \eta_A^\frown\infty \preceq \tau^\frown\infty \preceq\tau_\Pi.\]
A subrequirement $\eta_B,m_B$ might be set up after $e$'s $B$-use is enumerated, but before stage $s_1$, when the new $B$-use is picked. Like above, $e$'s $A$ and $C$-uses might be small enough to injure both sides of $\eta_B,m_B$, so we either want to lift $\theta(m_B)$ beyond $e$'s $B$-use, or lift one of $e$'s $A$- or $C$-uses. The former is not possible because $\eta_B\prec\tau$, so we are forced to lift $e$'s use again.

Suppose we fix the convention of not lifting $e$'s $A$-use again, choosing instead to lift the $B$ or $C$-uses when needed. In particular, we lift the $C$-use at some stage $s_2>s_1$. What happens if there is an $\eta_C$-gate of higher priority than $\eta_B$:
\[\eta_C^\frown\infty \preceq \eta_B^\frown\infty \preceq \eta_A^\frown\infty \preceq \tau^\frown\infty \preceq\tau_\Pi.\]
Repeating earlier argument, a subrequirement of $\eta_C$ that is set up after $e$'s $C$-use is enumerated, but before stage $s_2$, will call for $e$'s $B$ or $C$-use to be lifted after stage $s_2$. Even if we chose to lift $e$'s $A$-use instead of its $C$-use at stage $s_2$, we cannot avoid needing to lift one of $e$'s uses again after stage $s_2$, if there is a gate which works for $B\cap C\leq A$ of higher priority than $\eta_B$.

More generally, assume there are $n$ many gates of higher priority than $\tau$:
\[\eta_1^\frown\infty \preceq\eta_1^\frown\infty \preceq \ldots \preceq\eta_n^\frown\infty \preceq\tau.\]
Following the sketch above, to prevent $e$ from being un-realized, we lift its $A$-use once and set a restraint on $A$. Then by the convention of not lifting the $A$-use again, we alternate between lifting the $B$ and $C$-uses to appease the gates $\preceq\tau$. After $k$ many lifts, we would have appeased the $k$ lowest priority gates
\[\eta_{n-k+1} \prec\ldots \prec\eta_n,\]
and we say that $e$ has \emph{passed} those gates. $e$ passes one more higher priority gate with every $B$ or $C$ lift, which means that no more $n\leq|\tau|$ many $B$ and $C$ lifts will allow $e$ to pass all gates.

How do we ensure that a passed gate never gets into trouble with $e$ again? Revisit the earlier situation:
\[\eta_C^\frown\infty \preceq \eta_B^\frown\infty \preceq \eta_A^\frown\infty \preceq \tau^\frown\infty \preceq\tau_\Pi.\]
Assume $e$ passed $\eta_A$ then $\eta_B$, and has just enumerated its $B$-use, say at stage $s_0$, in an effort to pass $\eta_C$ also. When should $e$ pick its new use, in order to work with all subrequirements $\eta_B,m$ of $\eta_B$, regardless of how late that subrequirement is set up? By earlier argument, we cannot pick the new use before the next $\eta_C$-expansionary stage, if we want $e$ to pass $\eta_C$ later. Yet, we cannot wait beyond the next $\tau$-expansionary stage, according to the $\tau$-strategy. If $\eta_B,m$ is set up after stage $s_0$, it would notice $e$'s relatively small $A$ and $C$-uses, which will injure both sides of the subrequirement's computations if the uses are enumerated simultaneously later. To foresee this enumeration, $\eta_B,m$ will need to set $\theta(m)$ above the new $B$-use. Therefore, the new use must be picked by the next $\eta_B$-expansionary stage, not just by the next $\tau$-expansionary stage.

Generally, if $e$ has already passed a gate $\eta$, and has just enumerated one of its uses in an effort to pass a higher priority gate $\eta'\preceq\eta$, then the new use must be picked by the next $\eta$-expansionary stage, but not before the next $\eta'$-expansionary stage. To these ends, we shall do the picking at the next $\eta'$-expansionary stage.

Now consider the effect of $J$ on $\tau_\Pi$. Whenever $\tau_\Pi$ is waiting for an event to occur, its uses, if defined, need to extend traces because of $J$. We need to ensure that we can always choose targets for the extensions so as not to un-realize $e$, or to undo the effect from lifting uses. Consider the extension of the first $A$-use. Between knowing the $B$-restraint and changing the restraint to an $A$-restraint, the $A$-use and its traces will be enumerated. Until we know which $B$-values to avoid, the $A$-use's traces cannot be targeted for $B$. By $J_A$ and $J_C$, the $A$-use can only grow an $ACAC$-trace in the meantime. After this trace is enumerated, the $B$-restraint is dropped. Then a permanent $A$-restraint is imposed, and a new $A$-use which exceeds the restraint is picked. Since the $B$- and $C$-uses may be enumerated later, before knowing the $A$-values to avoid, the traces extending the $B$- or $C$-uses cannot be targeted for $A$, which forces us to extend them to $BCBC$- and $CBCB$-traces respectively.

Summarizing, before we alternate between lifting $B$- and $C$-uses, the traces associated with these uses involve small $B$- and $C$-elements. Wlog, suppose we wish to lift the $B$-use next to appease a gate with $B$ on one side. If we jump directly into lifting, small $B$-elements will still be present in the $C$-use's $CBCB$-trace, which when enumerated later, will undo the effect of the lifting. Therefore, prior to lifting, we change the traces of the $C$-use to a $CACA$-one to avoid small $B$-elements. We enumerate the $C$-use's trace except the use itself, then re-target with $A$ and $C$-balls. We shall refer to this process as the \emph{re-targeting of $B$'s extension}. Similarly, to appease gates that want the $C$-use to be lifted, we cannot allow the $B$-use's trace to contain small $C$-elements, and therefore \emph{re-target $B$'s extension} to a $BABA$-one prior to lifting.

What about the traces of $e$'s $A$-use? They also cannot be contain small $B$ or $C$-values. After re-targeting the $B$- and $C$-extensions, we want to lift the $B$-use next, so beforehand we need to change $A$'s trace to one that avoids $B$. But, we also know that after $B$ is lifted, we want to lift $C$ and will need the $A$-use to avoid extending to $C$-traces that are smaller than the $C$-elements of the $C$-use. For these reasons, we re-target $A$'s extension to the trace associated with the $C$-use itself. In other words, if currently the $A$-use's trace is $ax_0x_1\ldots$, and the $C$-use's trace is $ca_0c_0\ldots$, then we change $A$'s trace to $aca_0c_0\ldots$ by enumerating $x_0x_1\ldots$, and then re-targeting $a$ to $c$ just after $x_0$ is enumerated. This way, after the $B$-use is lifted, there will be no small $B$-elements in the traces of the $A$- or $C$-uses to undo the effect of the lifting. When the new $B$-use is picked, we extend it to a $BABA$-trace to avoid $C$. Similarly, after the $C$-use is lifted, we re-target $A$'s extension to the trace of the $B$-use, and after the new $C$-use is picked we extend it to a $CACA$-trace.

Now we bound the fickleness for satisfying $\tau_\Pi$. After $e$ is realized, we lift the $A$-use and re-target the $B$-, $C$-, then $A$-extensions. Each lifting or re-targeting requests $\leq\omega^{|\tau_\Pi|}$ many fickleness, if we follow the trace-extension and partitioning rules of Section~\ref{sec:necessary}. After re-targeting extensions, we alternate between lifting the $B$- and $C$-uses, lifting $\leq|\tau|$ many times. Finally, we enumerate $e$ with its uses, requesting one more fickleness. Thus a sufficient amount of fickleness for $\tau_\Pi$ is:
\begin{align} \label{eq:tau-fickleness}
    \omega^{|\tau_\Pi|}\cdot (4+|\tau|) +1 <\omega^{|\tau_\Pi|+1} <\omega^\omega.
\end{align}

Now a single follower may not be granted enough permissions by $\Delta(X)$. So we keep picking new followers to eventually succeed on one of them, threatening to approximate $\Delta(X)$ by a $<\omega^{\omega}$-fickle function $f$ with mind-change function $o$ if no follower succeeds. Letting $e_i$ denote the $i$-th follower of $\tau_\Pi$, $o(i)[s]$ shall bound the total length of the traces belonging to any of $e_0,\ldots,e_i$ that has yet to be enumerated by stage $s$. Function $o(i)[-]$ will then be non-increasing, decreasing strictly when some follower $e\leq e_i$ enumerates elements. By the earlier argument, each follower has traces of length $\leq\omega^{|\tau_\Pi|+1}$; therefore, $o(i)[0]\leq\omega^{|\tau_\Pi|+1}\cdot i$, making $f$ a $\leq\omega^{|\tau_\Pi|+2}$-fickle function. Function $f(i)[s]$ shall represent the guess of $\Delta(X,i)$ at stage $s$ and is, therefore, a $\leq\omega^{|\tau_\Pi|+2}$-fickle function via $o$.

Ideally, $e_i$ should receive permission only from $\Delta(X,i)$, but a complication arises when permission is received --- we need to cancel lower priority followers $e_{i'}>e_i$ to avoid undesirable interactions with gates $\eta,m$: $e_i$ might receive permission to act and injure one side of the gate; then $e_{i'}$ might receive permission and injure the other side, resulting in $\theta(m)$ being lifted and un-realizing $e_i$. But $e_{i'}$ might have been the follower to receive sufficient permissions. To prevent these permissions from being lost after $e_{i'}$'s cancellation, we adopt \citep{downey2007totally}'s strategy of letting $e_i$ inherit the fickleness meant for $e_{i'}$. Then since $\Delta(i')$ was fickle enough for $e_0,\ldots,e_{i'}$, $\Delta(i')$ will be fickle enough for $e_i$.

The above example illustrates another complication from multiple followers: Without knowing which follower to wait for, gates cannot foresee if both of its sides will be simultaneously injured. Therein lies the purpose a gate's computing sets --- whenever a trace $t$ is waiting to be enumerated, we put $t$ into a \emph{permitting} bin, and \emph{assign} $t$ \emph{permitting ball(s)}, which are element(s) targeted for the computing set(s) of gates, so that gates can use their computing set(s) to foresee the enumeration, and delay believing computations until those traces are gone. Suppose a trace $t$ is placed in the permitting bin, and $t$ contains elements targeted for $A$ and $B$ only. Since there are three types of gates, and together their computing sets include not just $A$ and $B$ but also $C$, we need to assign $t$ a $C$-permitting ball, which will given the same permitting code as $t$ to be enumerated with the trace. When some $\eta_C,m$ notices $t$ in the bin, $\eta_C$ will set $\theta(m)$ to exceed the $C$-ball, so as to $C$-computably know if $t$ will be enumerated later. Likewise, $\eta_A$ can use the $A$-elements in $t$ to $A$-computably know if $t$ will be enumerated, and similarly for $\eta_B$.

We can now provide the $\tau_\Pi$-strategy for the non-existence of $A\cap B$. For the other types of non-meets, swap the roles of $A$, $B$, and $C$ in the description below.
\begin{framed}
    \noindent\textbf{$\tau_\Pi$-strategy}: If $\tau_\Pi\succeq \tau^\frown\fin$, do nothing. Otherwise, \emph{initialize} $\tau_\Pi$ by cancelling all followers and restraints, if any, that have been set up by $\tau_\Pi$.
    \begin{enumerate}
        \item \emph{Set up}: Let $e$ be the largest follower that has been set up by $\tau$ and that has not been assigned to any other child-requirement. $e$'s $A$, $B$, and $C$-uses $\lambda_A(e)$, $\lambda_B(e)$, and $\lambda_C(e)$ would have just been assigned. Extend the traces of these uses, following the discussed rule that the traces of the $A$-use avoid $B$ and the traces of the $B$ and $C$-uses avoid $A$. Set the codes of all traces to $\delta\restriction(k+1)$, where $k$ is the number of followers of $\tau_\Pi$ so far.
        
        \item \emph{Realization}: Wait for $e$ to be realized at a $\tau_\Pi$-stage. Set a restraint on the $B$-side to stop $e$ from being un-realized.
        
        \item \emph{Lift $A$-use}: Enumerate the $A$-use's trace. If enough permissions are received that the $A$-use becomes enumerated, then at the next $\tau$-expansionary stage, the injured computation $V\restriction (\pi(e)+1) =\Gamma_A\restriction (\pi(e)+1)$ will have returned, and $\tau$ will pick a new large $A$-use $\lambda_A(e) >\gamma_A\restriction(\pi(e)+1)$. Set a restraint on the $A$-side to stop $e$ from being un-realized and cancel the earlier $B$-restraint. The $B$- and $C$-uses of $e$ can now grow traces targeted for $A$ if they wish.
        
        \item \emph{Re-target $B$-extension}: Let $bc_0\ldots$ be the current trace of the $B$-use. Enumerate this trace up to and including $c_0$. If enough permissions are granted that $c_0$ becomes enumerated, extend the trace of $b$ by a $BABA$-one.
        
        \item \emph{Re-target $C$-extension}: Let $cb_0\ldots$ be the current trace of the $C$-use. Enumerate this trace up to and including $b_0$. If enough permissions are granted that $b_0$ becomes enumerated, extend the trace of $c$ by a $CACA$-one.
        
        \item \emph{Re-target $A$-extension}: Let $ax\ldots$ be the current trace of the $A$-use. Enumerate this trace up to and including $x$. If enough permissions are granted that $x$ becomes enumerated, let the trace of $a$ extend by the $B$-use's $BABA$-trace so that the $A$-trace is $\lambda_A(e)\lambda_B(e)\ldots$. As discussed before, we consider $e$ not to have passed any infinitary gate of higher priority than $\tau$.
        
        \item \emph{Pass one more gate $\prec\tau$}: Let $\eta^\frown\infty \preceq\tau$ be the lowest priority gate that $e$ has not yet passed. If $\eta$ does not exist, then $e$ has passed all gates and we can go directly to the next diagonalization phase. Otherwise, let $Z\in\{B,C\}$ be any set that appears on one side of $\eta$. We want to lift the $Z$-use. Enumerate the $Z$-trace of $e$. If enough permissions are granted that the $Z$-use is enumerated, pick the new large $Z$-use at the next $\eta$-expansionary stage. If the trace extending $e$'s $A$-use is the previous $Z$-use, then after that previous use was enumerated, re-target $A$'s extension to the other use of $e$. Consider $e$ to have passed gate $\eta$, and return to the beginning of this gate-passing phase to pass the next lowest priority infinitary gate.
        
        \item \emph{Diagonalize}: Put $e$ and the traces of its uses into the permitting bin, assigning permitting ball(s) if necessary. Wait for a single $X$-permission to enumerate all elements simultaneously. Cancel all other followers and declare $\tau_\Pi$ satisfied by $e$.
    \end{enumerate}
\end{framed}

Apart from the rules above, we keep the usual conventions in priority constructions of cancelling lower priority followers whenever a follower $e$ receives attention, and of waiting for the next $\tau_\Pi$-stage before allowing $e$ to receive attention again and letting $\tau_\Pi$ pick a new follower. Also, if $e$ received permission to enumerate, we let $e$ inherit the permissions of all cancelled $e'>e$. Putting all strategies together, the overall construction is:
\begin{framed} \label{pg:construction}
    \noindent \textbf{Construction.} \emph{Stage 0.} Initialize all requirements.
    
    \emph{Stage $s+1$}: Let $\lambda_s\in\Lambda$ be the node of length $s$ representing the outcomes of the first $s$ requirements, where we set up a node $\lambda\preceq\lambda_s$ along the way if $\lambda$ has not been been set up. Let $\lambda\preceq\lambda_s$ or $\lambda<_L\lambda_s$ be the highest priority positive node that wants to act. Let $\lambda$ act and initialize all nodes of lower priority than $\lambda$.
\end{framed}

To prove that all requirements are satisfied, let $g\in [\Lambda]$ be the true path on the tree of construction $\Lambda$. We prove by induction on $n$, the length of the node along the true path, that the requirement represented by node $\lambda=g(n)$ is satisfied. In the following lemmas, we can assume from induction that we are working in stages after nodes to the left of $\lambda$ are never visited again, and after all $\rho\preceq\lambda$ have stopped inflicting injury. Then $\lambda$ will never be cancelled and will always have highest priority to act.

\begin{Lemma} \label{lemma:tau}
    If $\lambda=\tau$, then $\tau$ is satisfied and inflicts no injury.
\end{Lemma}
\begin{proof}
    $\tau$ never sets restraints or enumerates elements, and so cannot inflict injury. If the outcome of $\tau$ is $\fin$, then $\tau$ is vacuously satisfied. So, assume there are infinitely many $\tau$-expansionary stages. Each follower $e$ of $\tau$ is assigned to at most one child node $\tau_\Pi$, which by the $\tau_\Pi$-strategy acts only finitely often until cancellation. Uses $\lambda_A(e)$, $\lambda_B(e)$, and $\lambda_C(e)$ always exist at $\tau$-expansionary stages, and when $e$ is enumerated, those uses are also enumerated. Therefore, the uses are lifted finitely often, and $E(e)=\Lambda_A(e)=\Lambda_B(e)=\Lambda_C(e)$.
\end{proof}

\begin{Lemma} \label{lemma:tau-pi}
    If $\lambda=\tau_\Pi$, then $\tau_\Pi$ is satisfied and inflicts finite injury.
\end{Lemma}
\begin{proof}
    Wait for $\tau_\Pi$ to stop being initialized. Followers that are realized never become un-realized because of the restraints. If some follower was never realized, the lemma holds trivially. So assume all followers eventually become realized. Each follower $e$ requests $\leq\omega^{|\tau_\Pi|+1}$ many permissions. Some $e$ must eventually receive enough permissions to satisfy $\tau_\Pi$, otherwise we can approximate $\Delta(X)$ by a $\leq\omega^{|\tau_\Pi|+2}$-fickle function, contradicting the $\geq\omega^\omega$-fickleness of $\Delta(X)$. Even though $\omega^{|\tau_\Pi|+1} \geq\omega$, in reality $e$ only enumerates finitely many elements, except these elements may not be indexed canonically by a natural number. Therefore, the injury by $\tau_\Pi$ is finite.
\end{proof}

The $G$ and $J$ requirements are global and satisfied by construction, like in \citep{downey2020hierarchy}. The verification for $\rho$ also remains unchanged, because the requirement is of finite injury, which allows it to mix with $\tau$ and $\tau_\Pi$, for they inflict finite injury. The verification that takes the most work is for $\eta$:

\begin{Lemma} \label{lemma:eta}
    Let $\lambda=\eta$ and consider the subrequirement of $\eta$ at $m$. Let $s_0$ be the first $\eta$-stage where $\length(\eta)>m$. Let $s_1\geq s_0$ be the first $\eta$-stage where no follower that existed at stage $s_0$ has balls targeted for the computing set of $\eta$ that will later be enumerated. Note that the computing sets of $\eta$ can determine stage $s_1$. From stage $s_1$ onward, we say that $\eta,m$ \emph{believes its computation}.
    \begin{enumerate}
        \item The subrequirement can only be injured by followers that existed at stage $s_0$. Therefore $\theta(m)$ is lifted finitely often, since followers act finitely often.
        \item After stage $s_1$, $\Phi_0(m)$ and $\Phi_1(m)$ cannot be simultaneously injured.
    \end{enumerate}
    If $\eta^\frown\fin \prec g$, then $\eta$ is vacuously satisfied. If $\eta^\frown\infty \prec g$, then $\eta$ will also be satisfied from the above two claims.
\end{Lemma}
\begin{proof}
    Claim 1 follows from the tree framework: Let $y$ be the first element whose enumeration contradicts the claim, and let $\sigma$ be node to which $y$ belongs. If $\sigma<_L\eta$ or $\sigma\preceq\eta$, then $y$ will never be enumerated otherwise $\eta$ will be cancelled; if $\sigma>_L\eta^\frown\infty$ then $y$ will be cancelled at expansionary stages, and will be too large to inflict injury at non-expansionary stages. So it must be that $\sigma\succeq \eta^\frown\infty$. Before $y$ acted, the subrequirement's use must have been lifted due to a recent action by some other $y'$. $y'$ must have existed at stage $s_0$ by the choice of $y$ being the first. Now $y$ must be picked after the action by $y'$, otherwise if $y$ is of higher priority than $y'$ then $y'$ would be cancelled before it could act, and if $y$ is of lower priority then $y$ would be cancelled by $y'$. Since $\sigma\succeq \eta^\frown\infty$, $y$ must be picked at an $\eta$-expansionary stage, which is a stage after $\eta,m$ has recovered from the injury inflicted by $y'$. Then $y$ be too large to inflict injury.
    
    Claim 2: We work in stages after $s_1$. By Claim 1, we can ignore all followers that are set up after stage $s_0$ or that were cancelled by stage $s_1$. Assume, for a contradiction, that there is a first stage where $\eta,m$ was injured on both sides simultaneously. There are three cases for how the injury could be inflicted, and we show that none of them are possible.
    
    Case 1 --- The injury was due to different followers: The permitting balls prevent this situation. Assume follower $y_0$ injured one side of $\eta,m$, but before the side recovered, some other $y_1$ belonging to $\sigma\succeq \eta^\frown\infty$ injures the other side. $y_0$ must have lower priority than $y_1$; otherwise, $y_1$ would be cancelled when $y_0$ attacked, and Claim 1 prevents followers that are set up later from inflicting injury. By our choice of $s_1$, $y_1$ did not have a permitting ball at stage $s_1$. $y_1$ can only be assigned permitting balls after $y_0$'s enumeration; otherwise, $y_0$ would be cancelled, and cannot injure $\eta,m$ by Claim 1. But after the assignment, $y_1$ would have waited for a $\sigma$-stage, which is an $\eta$-expansionary stage, before enumerating. The injury from $y_0$ would have recovered before $y_1$ could attack.
    
    Case 2 --- The injury was from a follower of $\rho$: This is not possible for the same reason that $\rho$ requirements work well with $\eta$ requirements in the embedding of $M_3$, which has been discussed in detail in \citep{downey2020hierarchy}.

    Case 3 --- The injury was from a follower $e$ of $\tau_\Pi$: Until we reach the diagonalization phase of the $\tau_\Pi$-strategy, enumerations by $e$ work like in Case 2, which does not give problems. So the injury must be inflicted at the final phase, when $e$'s uses are simultaneous enumerated. Assume wlog that $\tau_\Pi$ works for the non-existence of $A\cap B$. If $\eta\succeq\tau^\frown\infty$, then $\eta,m$ will always know the updated uses of $e$, since $e$ was set up before $\eta,m$ and always picks uses by $\tau$-expansionary stages. Therefore, $\eta,m$ knows $e$ will be enumerated and would not have believed computations until afterwards. So, assume $\eta^\frown\infty \preceq\tau$, and wait for $e$ to pass $\eta$, which would involve lifting the $B$- or $C$-use of $e$. Wlog, assume it was the $B$-use that was lifted, which means that $B$ appeared on exactly one side of $\eta$. If $\eta,m$ was set up before the passing, then after passing, new $B$-elements associated with $e$ will be too large to injure the $B$-side of $\eta,m$, by the $\tau_\Pi$-strategy of not letting $e$'s $A$- or $C$-traces contain small $B$-elements. Therefore the simultaneous injury cannot injure both sides of $\eta,m$. So it must be that $\eta,m$ was set up after the passing. Then $\eta,m$ can keep track of the $A$- and $C$-uses of $e$, because $e$'s $A$-use is never lifted again and its $C$-use, if lifted after passing $\eta$, will always be picked by the next $\eta$-expansionary stage. In particular, $\eta,m$ can keep lifting $\theta(m)$ above the relevant use of $e$, which will help the subrequirement delay believing computations until after $e$ is enumerated.
\end{proof}

\section{Future Work} \label{sec:future}
We can proceed lattice theoretically or degree theoretically. We can consider using lattice theory to find more $>\omega^2$-candidates:
\begin{Question}
    Can we systematically list the non-distributive lattices and apply Corollary~\ref{cor:non-candidate} to get $>\omega^2$-candidates?
\end{Question}

Given a candidate, like the one in Figure~\ref{fig:peter}, we can extend the known degree theoretic techniques to characterize the lattice. The layering and trace-extension methods need to be generalized before they can be applied here.

\begin{Conjecture}
    The lattice in Figure~\ref{fig:peter} characterizes $\geq\omega^\omega$-fickleness.
\end{Conjecture}

\begin{figure}[tpb]
    \centering
    \begin{tikzpicture}[every node/.style={circle,inner sep=1.2pt}]
        \tikzset{omega2/.pic ={
            \node[circle,fill=black] (0) at (0,-1) {};
            \node[circle,fill=black] (1) at (0,1.5) {};
            \node[circle,fill=black] (A) at (-0.5,-0.5) {};
            \node[circle,fill=black] (X) at (0.5,-0.5) {};
            \node[circle,fill=black] (AX) at (0,0) {};
            \node[circle,fill=black] (B) at (-1,0) {};
            \node[circle,fill=black] (Y) at (1,0) {};
            \node[circle,fill=black] (BX) at (-0.5,0.5) {};
            \node[circle,fill=black] (AY) at (0.5,0.5) {};
            \node[circle,fill=black] (C) at (-1.5,0.5) {};
            \node[circle,fill=black] (Z) at (1.5,0.5) {};
            \draw [-,thick] (0) --(A) --(B) --(C) --(1) --(Z) --(Y) --(X) --(0);
            \draw [-,thick] (1) --(AY) --(AX) --(A);
            \draw [-,thick] (1) --(BX) --(AX) --(X);
            \draw [-,thick] (BX) --(B);
            \draw [-,thick] (AY) --(Y);
        }}
        \path (0,0) pic[scale=0.5]{omega2};
    \end{tikzpicture}
    \caption{A $>\omega^2$-candidate, proposed by Peter Cholak.}
    \label{fig:peter}
\end{figure}
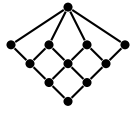



\bibliographystyle{elsarticle-harv} 
\bibliography{main}


%
%
%
\end{document}